\documentclass[10pt]{amsart}

\newtheorem{theoreme}{Theorem}[section]

\newtheorem{lemme}[theoreme]{Lemma}

\newtheorem{proposition}[theoreme]{Proposition}

\newtheorem{remarque}[theoreme]{Remark}

\newtheorem{corollaire}[theoreme]{Corollary}

\input xy

\xyoption{all}

\usepackage{amscd}

\usepackage{amssymb}

\usepackage[a4paper, left=2cm,right=2cm,bottom=2cm,top=3.5cm]{geometry}
\usepackage{amsthm,array,amssymb,amscd,amsfonts,latexsym, url}
\usepackage{amsmath}
\usepackage{graphicx,epsfig}

\title{Remarks on the $\mathrm{CH}_2$ of cubic hypersurfaces}
\author{Ren\'e Mboro}
\date{}
\begin{document}

\begin{abstract}
This paper presents two approaches to reducing problems on $2$-cycles on a smooth cubic hypersurface $X$ over an algebraically closed field of characteristic $\neq 2$, to problems on $1$-cycles on its variety of lines $F(X)$. The first one relies on osculating lines of $X$ and Tsen-Lang theorem. It allows to prove that $\mathrm{CH}_2(X)$ is generated, via the action of the universal $\mathbb P^1$-bundle over $F(X)$, by $\mathrm{CH}_1(F(X))$. When the characteristic of the base field is $0$, we use that result to prove that if $dim(X)\geq 7$, then $\mathrm{CH}_2(X)$ is generated by classes of planes contained in $X$ and if $dim(X)\geq 9$, then $\mathrm{CH}_2(X)\simeq \mathbb Z$. Similar results, with slightly weaker bounds, had already been obtained by Pan(\cite{xua_pan}). The second approach consists of an extension to subvarieties of $X$ of higher dimension of an inversion formula developped by Shen (\cite{Shen_main}, \cite{Shen_univ}) in the case of $1$-cycles of $X$. This inversion formula allows to lift torsion cycles in $\mathrm{CH}_2(X)$ to torsion cycles in $\mathrm{CH}_1(F(X))$. For complex cubic $5$-folds, it allows to prove that the birational invariant provided by the group $\mathrm{CH}^3(X)_{tors,AJ}$ of homologically trivial, torsion codimension $3$ cycles annihilated by the Abel-Jacobi morphism is controlled by the group $\mathrm{CH}_1(F(X))_{tors,AJ}$ which is a birational invariant of $F(X)$, possibly always trivial for Fano varieties.
\end{abstract}
\maketitle
\section*{Introduction}
Let $X\subset \mathbb P^{n+1}_{\mathbb C}$ be a smooth hypersurface of degree $d\geq 2$.
Let $F_r(X)\subset G(r+1,n+2)$ be  the variety of $\mathbb P^r$'s contained in $X$ and $P_r=\mathbb{P}(\mathcal{E}_{r+1\mid F_r(X)})\subset F_r(X)\times X$ be  the  universal $\mathbb P^r$-bundle. One has the incidence correspondence
$$p_r:P_r\rightarrow  F_r(X),\,\,q_r:P_r\rightarrow X.$$
We will be particularly interested in this chapter in the cases $r=1$ and $r=2$, $d=3$.
It is known (see for example \cite{Es_Le_Vie}, \cite{vois_lect}) that if $X$ is covered by projective spaces of dimension $1\leq r< \frac{n}{2}$, that is $q_r$ is surjective,  then $\mathrm{CH}_i(X)_{\mathbb Q} \simeq\mathbb Q$ for $i< r$ and for $\frac{n}{2}> i\geq r$,  there is an inversion formula implying that  $$P_{r,*}:\mathrm{CH}_{i-r}(F_r(X))_{hom,\mathbb Q}\rightarrow\mathrm{CH}_i(X)_{hom,\mathbb Q}$$
is surjective.
We recall briefly how it works: Up to taking a desingularization and general hyperplane sections of $F_r(X)$, we can assume that $F_r(X)$ is smooth and $q_r$ is generically finite of degree $N>0$. Let $H_X=c_1(O_X(1))\in{\rm CH}^1(X)$
and $h=q_r^*H_X\in {\rm CH}^1(P_r)$. Given a cycle $\Gamma\in \mathrm{CH}_i(X)_{hom}$,   we have $N\Gamma=q_{r*}q_r^*\Gamma$ and we can write $q_r^*\Gamma=\sum_{j=0}^rh^j\cdot p_r^*\gamma_j$ where  $\gamma_j\in \mathrm{CH}_{i+j-r}(F_r(X))_{hom}$. Now, $d q_{r*}(h^j\cdot p_r^*\gamma_j)=dH_X^j\cdot q_{r*}(p_r^* \gamma_j)=0$ for $j>0$ since
 $$dH_X\cdot =i_X^*i_{X,*}:\mathrm{CH}_l(X)_{hom}\rightarrow \mathrm{CH}_{l-1}(X)_{hom},$$
 where $i_X$ is the inclusion of $X$ into $\mathbb{P}^{n+1}$,
 factors through $\mathrm{CH}_l(\mathbb P^{n+1}_{\mathbb C})_{hom}$ hence is zero. So we get
 $$dN\Gamma=q_{r*}p_r^*(d\gamma_0)$$
  which gives $\mathrm{CH}_i(X)_{hom, \mathbb Q}=0$ for $i<r$ since in this range $\mathrm{CH}_{i-r}(F(X))=0$,
  and more generally the desired surjectivity. Working a little more, this method gives,
  in the case of $2$-cycles on cubic fivefolds, the following
  result (which is a precision of \cite{Es_Le_Vie}, \cite{otwinowska}):
  \begin{proposition}\label{propajtrivial} Let $X$ be a smooth cubic fivefold. Then
  the kernel of the Abel-Jacobi map $\mathrm{CH}_2(X)_{AJ}:={\rm Ker}\,(\Phi_X:{\rm CH}_2(X)_{hom}\rightarrow J^5(X))$ is of $18$-torsion.
  \end{proposition}
\begin{proof} For cubic hypersurfaces of dimension $\geq 3$, after taking hyperplane sections of $F_1(X)$, the degree of the generically finite morphism $P_1\rightarrow X$ is $6$.
If $\Gamma\in{\rm CH}_2(X)_{AJ}$, we can use the fact that
$3H_X\cdot \Gamma=0$ in ${\rm CH}^4(X)$ and
thus (using the notation $P,\,p,\,q$ in this case)
\begin{eqnarray}
\label{eqintro3jan}
3h\cdot q^*\Gamma= 3(h\cdot p^*\gamma_0+h^2\cdot p^*\gamma_1)=0
\end{eqnarray}
in ${\rm CH}^4(P)$. As $h^2=h\cdot p^*l-p^*c_2$ in ${\rm CH}^2(P)$, where $l$ and $c_2$ are the natural
Chern classes on $F_1(X)$ restricted from the Grassmannian,
we deduce from (\ref{eqintro3jan}):
$$3\gamma_0=-3l\cdot \gamma_1 \,\,{\rm in}\,\,{\rm CH}^3(F_1(X)).$$

Combining this with the previous  argument then gives
$18\Gamma=q_{*}p^*(3\gamma_1\cdot l^{3})$ where $3\gamma_1$ is a codimension $2$-cycle homologous to $0$ and Abel-Jacobi equivalent to $0$ on  $F_1(X)$. Finally we conclude using \cite[Theorem 1 (i)]{Bl-Sr} and the fact that $F_1(X)$ is rationally connected, which implies that
${\rm CH}^2(F_1(X))_{AJ}=0$.
\end{proof}
\indent
 The denominators appearing in the above argument do not allow to understand $2$-torsion cycles.
 On the other hand, as smooth cubic hypersurfaces admit a degree $2$ unirational parametrization (\cite{Cl-Gr}), all functorial birational invariants are $2$-torsion so that, for functorial birational invariant constructed using torsion cycles, the above method gives no interesting information.
 Our aim in this chapter is to give inversion formulas with integral coefficients,
 allowing in some cases to also control the torsion of the group of  cycles, which is especially important for those hypersurfaces in view of rationality problems.

 In this chapter, we present two approaches to study the surjectivity of the map
  $P_{1*}$ on cycles with integral coefficient  for cubic hypersurfaces.
  The first one is presented in the first section and uses the osculating lines of $X$; it gives the following result:
\begin{theoreme}\label{thm_intro1} Let $X\subset \mathbb P^{n+1}_k$, with $n\geq 2^i+1$ be a smooth cubic hypersurface over an algebraically closed field $k$ of characteristic not equal to $2$, containing a linear subspace of dimension
$i<\frac{n}{2}$. Assuming resolution of singularities in dimension $\leq i$, $P_{1,*}:\mathrm{CH}_{i-1}(F_1(X))\rightarrow \mathrm{CH}_i(X)$ is surjective.
\end{theoreme}
In the case where $i=2$, the theorem associates to any $2$-cycle a $1$-cycle on $F_1(X)$. As, for $i=2$, the condition to apply the theorem is $dim_k(X)\geq 5$, $F_1(X)$ is a smooth Fano variety hence separably rationally
 connected in characteristic $0$.  By work of Tian and Zong (\cite{Tian-Zong}), $\mathrm{CH}_1(F_1(X))$ is then generated by  classes of rational curves. A direct consequence is the following:
\begin{corollaire}\label{cor_intro1} Let $X\subset \mathbb P^{n+1}_k$ be a smooth cubic hypersurface over an algebraically closed field $k$ of characteristic $0$. If $n\geq 5$, then $\mathrm{CH}_2(X)$ is generated by cycle classes of rational surfaces.
\end{corollaire}
\begin{remarque}{\rm This result is true for a  different reason also in dimension $4$, see
Proposition \ref{provoisin}.}
\end{remarque}
In the second section, we study $1$-cycles on $F_1(X)$ in order to prove that, in some cases, we can take as generators of $\mathrm{CH}_1(F_1(X))$ only the ``lines'' i.e. the rational curves of degree $1$, of $F_1(X)$. We obtain the following result:
\begin{theoreme} Let $X\subset \mathbb P^{n+1}_k$ be a smooth hypersurface of degree $d$ over an algebraically closed field $k$ of characteristic $0$. If $\frac{d(d+1)}{2}<n$ and $F_1(X)$ is smooth then $\mathrm{Griff}_1(F_1(X))=0$. Moreover, $\mathrm{CH}_1(F_1(X))$ is generated by lines.
\end{theoreme}
This theorem has the following consequence in the case of cubic hypersurfaces:
\begin{corollaire}\label{prop_intro2} Let $X\subset \mathbb P^{n+1}_k$ be a smooth cubic hypersurface over an algebraically closed field $k$ of characteristic $0$. If $n\geq 7$, then $\mathrm{CH}_2(X)$ is generated by classes of planes $\mathbb P^2\subset X$ and therefore ${\rm CH}_2(X)_{hom}={\rm CH}_2(X)_{alg}$. If $n\geq 9$, then $\mathrm{CH}_2(X)\simeq \mathbb Z$.
\end{corollaire}

\begin{remarque}{\rm Some of the results of the first two sections had already been obtained by Pan (\cite{xua_pan}) in charateristic $0$ but with weaker bounds. For example for cubic hypersurfaces, he proves the surjectivity of $P_{1,*}:{\rm CH}_1(F_1(X))\rightarrow {\rm CH}_2(X)$ for $n\geq 17$, the fact that ${\rm CH}_1(F_1(X))_{hom}={\rm CH}_1(F_1(X))_{alg}$ for $n\geq 13$ and that ${\rm CH}_2(X)=\mathbb Z$ for $n\geq 18$} (see \cite[Theorem 1.2 and Proposition 2.2]{xua_pan}).
\end{remarque}

The last section is devoted to a second approach to the integral coefficient problem; it consists of a generalization of a formula developped by Shen (\cite{Shen_main}, see also \cite{Shen_univ}) in the case of $1$-cycles of cubic hypersurfaces. Let us introduce some notations. Let us denote $Y^{[2]}$ the Hilbert scheme of length $2$ subschemes of any variety $Y$. For a smooth cubic hypersurface $X$, let us denote $i_{P_2}:P_2\hookrightarrow X^{[2]}$ the subscheme of length $2$ subschemes supported on a line of $X$. The variety $P_2$ admits, by definition a projection $p_{P_2}:P_2\rightarrow F_1(X)$ associating to a length $2$ subscheme, the line it is supported on. We prove the following:
\begin{theoreme}\label{thm_intro2} Let $X\subset \mathbb P^{n+1}_k$ be a smooth cubic hypersurface over a field $k$, and $\Sigma$ a smooth subvariety of $X$ of dimension $d$. Then, there is an integer $m_\Sigma$ such that: $$(2deg(\Sigma)-3)\Sigma + P_{1,*}[(p_{P_2,*}i_{P_2}^*\Sigma^{[2]})\cdot c_1(\mathcal O_{F_1(X)}(1))^{d-1}]=m_\Sigma H_{X}^{n-d}$$ where $\mathcal O_{F_1(X)}(1)$ is the Pl\"ucker line bundle.
\end{theoreme}

 This inversion formula is more powerful than the first approach as it will allow us
  to lift, modulo $\mathbb Z\cdot H_X^{n-2}$, torsion $2$-cycles on $X$ to torsion $1$-cycles on $F_1(X)$. The application
  we have in mind is the study of certain birational invariants of $X$.
  When $k=\mathbb C$, it was observed in \cite{vois_d4_unr} that the group $\mathrm{CH}^3_{tors,AJ}$  of homologically trivial torsion codimension $3$ cycles annihilated by the Abel-Jacobi map is a birational invariant of smooth projective varieties which is trivial for stably rational varieties and more generally for varieties admitting a Chow-theoretic decomposition of the diagonal.
   This is a consequence of  the deep result  due to  Bloch (\cite{Bloch_tors}, \cite{CSS}) that the group $\mathrm{CH}^2(Y)_{tors,AJ}$ of homologically trivial torsion codimension $2$ cycles annihilated by the Abel-Jacobi map is $0$ for any smooth projective variety.  For cubic hypersurfaces, as already mentioned, it follows from the existence of a unirational parametrization of degree $2$
 that  $\mathrm{CH}^3(X)_{tors,AJ}$ is a $2$-torsion group.  Although we have not been able
 to compute this group, we obtain the following:

\begin{theoreme}\label{prop_eq} Let $X\subset \mathbb P^{n+1}_{\mathbb C}$ be a smooth cubic hypersurface, with $n\geq 5$. Then for any $\Gamma\in \mathrm{CH}_2(X)_{tors}$, there are a homologically trivial cycle $\gamma\in \mathrm{CH}_1(F_1(X))_{tors,hom}$ and an odd integer $m$ such that $P_*(\gamma)=m\Gamma$. \\
\indent Moreover, when $n=5$, starting from $\Gamma\in \mathrm{CH}_2(X)_{tors, AJ}=\mathrm{CH}^3(X)_{tors, AJ}$, we can find a $\gamma\in \mathrm{CH}_1(F(X))_{tors, AJ}$ such that $P_*(\gamma)=\Gamma$. In particular, if the $2$-torsion part of $\mathrm{CH}_1(F_1(X))_{tors, AJ}$ is $0$ then $\mathrm{CH}^3(X)_{tors,AJ}=0$.
\end{theoreme}
As a consequence of a theorem of Roitman (\cite{Roitman}) asserting that torsion $0$-cycles of any smooth projective variety $Y$ inject in $\mathrm{Alb}(Y)$, the group $\mathrm{CH}_1(F_1(X))_{tors, AJ}$
is a stable birational invariant of the variety $F_1(X)$ which is trivial for stably rational varieties or even for varieties admitting a Chow theoretic decomposition of the diagonal.\\
\indent The group $\mathrm{CH}^3(X)_{tors,AJ}$ has a quotient which has an interpretation in terms of
unramified cohomology.
We recall that, for a smooth complex projective variety $Y$ and an abelian group $A$, the degree $i$ unramified cohomology group $H_{nr}^i(Y,A)$ of $Y$ with coefficients in $A$ can be defined (see \cite{bloch_ogus}) as the group of global sections $H^0(Y,\mathcal H^i(A))$, $\mathcal H^i(A)$ being the sheaf associated to the presheaf $U\mapsto H^i(U(\mathbb C),A)$, where this last group is the Betti cohomology of the complex variety $U(\mathbb C)$.  The groups $H^i_{nr}(Y,A)$ provide stable birational invariants (see \cite{CT-O}) of $Y$, which vanish for projective space i.e. these groups are invariants under the relation: $$Y\sim Z\ if\ Y\times \mathbb P^r\ is\ birationally\ equivalent\ to\ Z\times \mathbb P^s\ for\ some\ r,\ s.$$
\indent Unramified cohomology group with coefficients in $\mathbb Z/m\mathbb Z$ or $\mathbb Q/\mathbb Z$ has been used in the study of L\"uroth problem, that is the study of unirational varieties which are not rational, to provide examples of unirational varieties which are not stably rational (see \cite{ar-mum},\cite{CT-O}). In the case of smooth cubic hypersurfaces $X\subset \mathbb P^{n+1}_{\mathbb C}$, since there is a unirational parametrization of degree $2$ of $X$ (see \cite{Cl-Gr}) and there is an action of correspondences on unramified cohomology groups compatible with composition of correspondences (see \cite[Appendice]{ct_vois}), the groups $H^i_{nr}(X, \mathbb Q/\mathbb Z)$, $i\geq 1$, are $2$-torsion groups. It is known that $H^1_{nr}(X,\mathbb Q/\mathbb Z)=0$ for $n\geq 2$ since this group is isomorphic to the torsion in the Picard group of $X$ (see \cite[Proposition 4.2.1]{ct_lect_unr}).\\
\indent Since for cubic hypersurfaces of dimension at least $2$, $H^2_{nr}(X,\mathbb Q/\mathbb Z)$ is equal to the Brauer group $Br(X)$ (see \cite[Proposition 4.2.3]{ct_lect_unr}), we have $H^2_{nr}(X,\mathbb Q/\mathbb Z)=0$.\\
\indent As for $H^3_{nr}(Y,\mathbb Q/\mathbb Z)$, it was reinterpreted in \cite[Theorem 1.1]{ct_vois} for rationally connected varieties $Y$ as the torsion in the group $Z^4:=Hdg^4(Y)/H^4(Y,\mathbb Z)_{alg}$, quotient of degree $4$ Hodge classes by the subgroup of $H^4(Y,\mathbb Z)$ generated by classes of codimension $2$ algebraic cycles, i.e. $H^3_{nr}(Y,\mathbb Q/\mathbb Z)$ measures the failure of the integral Hodge conjecture in degree $4$. For cubic hypersurfaces $X\subset \mathbb P^{n+1}_{\mathbb C}$, by Lefschetz hyperplane theorem, the only non trivial case where the integral Hodge conjecture could fail in degree $4$ is for cubic $4$-folds but it was proved to hold by Voisin in \cite{vois_4fold_ihc}.\\
\indent The group $H^4_{nr}(Y,\mathbb Q/\mathbb Z)$ was reinterpreted in \cite[Corollary 0.3]{vois_d4_unr} for rationally connected varieties $Y$ as the group $\mathrm{CH}^3(Y)_{tors,AJ}/alg$ of homologically trivial torsion codimension $3$ cycles annihilated by Abel-Jacobi map (or torsion codimension $3$ cycles annihilated by Deligne cycle map) modulo algebraic equivalence. For dimension reason $H^4_{nr}(X,\mathbb Q/\mathbb Z)=0$ for cubic hypersurfaces of dimension $\leq 3$. For cubic $4$-folds, since $H^4_{nr}(X,\mathbb Q/\mathbb Z)\simeq \mathrm{CH}^3(X)_{tors,AJ}/alg\simeq \mathrm{CH}_1(X)_{tors,AJ}/alg\subset \mathrm{Griff}_1(X)$, the work of Shen (\cite{Shen_main}) proves that $H^4_{nr}(X,\mathbb Q/\mathbb Z)=0$. The vanishing of $\mathrm{CH}^3(X)_{tors,AJ}\simeq \mathrm{CH}_1(X)_{tors,AJ}$ for cubic $4$-flods follows also essentially from the work of Shen (see Proposition \ref{Prop_ex_4-fold}). For a cubic $5$-fold $X$, by the choice of a $\mathbb P^2\subset X$ to project from, we see that $X$ is birational to a quadric bundle over $\mathbb P^3_{\mathbb C}$ so that by work of Kahn and Sujatha (\cite[Theorem 3]{KS}), $H^4_{nr}(X,\mathbb Q/\mathbb Z)=0$. Hence, for a cubic hypersurface $\mathrm{CH}^3(X)_{tors,AJ}\subset \mathrm{CH}^3(X)_{alg}$.\\
\textit{}\\

\section{First formula}
Let $X\subset \mathbb P^{n+1}_k$ be a smooth hypersurface of degree $d\geq 2$ and dimension $n\geq 3$ over an algebraically closed field $k$. Let us denote $F(X)\subset G(2,n+2)$ its variety of lines and $P\subset F(X)\times  X$ the correspondence given by the universal $\mathbb P^1$-bundle, and
 $$p:P\rightarrow F(X),\,\,q:P\rightarrow X$$
  the two projections. For a general hypersurface of degree $d\leq 2n-2$, $F(X)$ is a smooth connected variety (\cite[Theorem 4.3, Chap. V]{rat-cur-kol}).\\
\indent Let us denote $Q=\{([l],x)\in \mathbb P(\mathcal E_2),\   l\subset X\ \mathrm{or}\ l\cap X=\{x\}\}$ the correspondence associated to the family of osculating lines of $X$, and
 $$\pi:Q\rightarrow X,\,\, \varphi:Q\rightarrow G(2,n+2)$$ the two projections. We have $P\subset Q$.

We have the following easy lemma:
\begin{lemme}\label{prop_osc_sm} The fiber of $\pi:Q\rightarrow X$ (resp. $q:P\rightarrow X$) over any point $x$ in the image of $\pi$ (resp. of $q$) is isomorphic to an intersection of hypersurfaces of type $(2,3,\dots,d-1)$ in $P(T_{X,x})$ (resp. of type $(2,3,\dots,d)$). Moreover, for $X$ general, $Q$ is a local complete intersection subscheme of $\mathbb P(\mathcal E_2)$ of dimension $2n-d+1$. If $char(k)=0$, then $Q$ is smooth for $X$ general.
\end{lemme}
\begin{proof} By definition $Q$ is the set of $([l],x)$ in $\mathbb P(\mathcal E_2)$ over $G(2,n+2)$ where the restriction of the equation defining $X$ is $0$ or proportional to $\lambda_x^d$, where
$\lambda_x$ is
 the linear form defining  $x$ in $l$. Let $x\in X$ and $\mathcal P$ a hyperplane not containing $x$. There is an isomorphism $P(T_{\mathbb P^{n+1},x})\rightarrow \mathcal P$ given by $[v]\mapsto l_{(x,v)}\cap \mathcal P$, where $l_{(x,v)}$ is the line of $\mathbb P^{n+1}$ determined by $(x,v)$. We can assume that $x=[1,0,\dots,0]$ and $\mathcal P=\{X_0=0\}$. Let $l$ be a line through $x$ and $[0,Y_1,\dots,Y_{n+1}]\in \mathcal P$ the point associated to $l$. Then, denoting $f$ an equation defining $X$, since $x\in X$, we can write $f(X_0,\dots,X_{n+1})=\sum_{i=0}^{d-1}X_0^if_{d-i}(X_1,\dots,X_{n+1})$, where $f_i$ is a homogeneous polynomial of degree $i$. The general point of $l$ has
 coordinates
 $(\mu,\lambda Y_1,\dots,\lambda Y_{n+1})$ where $\lambda=\lambda_x$ and $\mu$ form a basis of linear
 forms on $l$. The restriction of $f$ to $l$ thus writes $\sum_{i=0}^{d-1}\mu^i\lambda^{d-i}f_{d-i}(Y_1,\dots,Y_{n+1})$.
  Thus the line $l$ is osculating if and only if $f_j(Y_1,\dots,Y_{n+1})=0$, $\forall j<d$. The first equation $f_1$
   is the differential of $f$ at $x$ and its vanishing hyeperplane is $P(T_{X,x})$, so we proved
   that $\pi^{-1}(x)$ is isomorphic to an intersection of hypersurfaces of type $(2,3,\dots,d-1)$ in $P(T_{X,x})$. We show likewise that the fiber  $q^{-1}(x)$ is  isomorphic to an intersection of hypersurfaces of type $(2,3,\dots,d)$.\\
\indent On the projective bundle $p_G:\mathbb P(\mathcal E_2)\rightarrow G(2,n+2)$, we have the exact sequence:
\begin{equation}\label{ex_seq_rel} 0\rightarrow \Omega_{\mathbb P(\mathcal E_2)/G(2,n+2)}(1)\rightarrow p_G^*\mathcal E_2\rightarrow \mathcal O_{\mathbb P(\mathcal E_2)}(1)\rightarrow 0
\end{equation}
The last morphism being the evaluation morphism, we see that $\Omega_{\mathbb P(\mathcal E_2)/G(2,n+2)}(1)_{([l],x)}$ is the ideal sheaf of $x$ in $l$. Taking the symmetric power of the dual of (\ref{ex_seq_rel}) yields the exact sequence: $$0\rightarrow {\rm Sym}^d(\Omega_{\mathbb P(\mathcal E_2)/G(2,n+2)}(1))\rightarrow p_G^*{\rm Sym}^d\mathcal E_2\rightarrow p_G^*{\rm Sym}^{d-1}\mathcal E_2\otimes \mathcal O_{\mathbb P(\mathcal E_2)}(1)\rightarrow 0$$ where the first morphism is the $d$-th symmetric power of the (first) inclusion in (\ref{ex_seq_rel}).\\
\indent Now, let $f$ be an equation defining $X$; it gives rise to a section $\sigma_f$ of $p_G^*{\rm Sym}^d\mathcal E_2$. Let $\overline{\sigma_f}$ be the section of $p_G^*{\rm Sym}^{d-1}\mathcal E_2\otimes \mathcal O_{\mathbb P(\mathcal E_2)}(1)$ induced by $\sigma_f$. Then the zero locus of $\overline{\sigma_f}$ is exactly the locus of $([l],x)$ where the restriction to $l$ of the equation defining $X$ is $0$ or equal to the linear form induced by $x$ on $l$ to the power $d$. So $Q$ is the zero locus in $\mathbb P(\mathcal E_2)$ of a section of the vector bundle $p_G^*{\rm Sym}^{d-1}\mathcal E_2\otimes \mathcal O_{\mathbb P(\mathcal E_2)}(1)$. As this vector bundle is globally generated, the zero locus of a general section is a local complete intersection (even regular if $char(k)=0$) subscheme of $\mathbb P(\mathcal E_2)$ of dimension $2n-d+1$.
\end{proof}

\begin{theoreme}\label{thm_form_osc} Let $X\subset \mathbb P^{n+1}_k$ be a smooth hypersurface of degree $d$ and let $P\subset F(X)\times X$ be the incidence
correspondence. Assume $\sum_{i=1}^{d-1} i^r\leq n$ with $r>0$ and, if $r>3$ and $char(k)>0$, assume resolution of singularities of varieties of dimension $r$. Then for any cycle $\Gamma\in \mathrm{CH}_r(X)$ there is a $\gamma\in \mathrm{CH}_{r-1}(F(X))$ such that $$d\Gamma + P_*(\gamma)\in \mathbb Z\cdot H_{X}^{n-r}$$ where $H_X=c_1(\mathcal O_{X}(1))$.
\end{theoreme}
\begin{proof} Let $\Sigma\subset X$ be an integral subvariety of dimension $r>0$.
By Tsen-Lang theorem (\cite{Lang}, \cite[Theorem 2.10]{starr_tsen}), the function field $k(\Sigma)$ of $\Sigma$ is $C_r$. As the fibers of $\pi:Q\rightarrow X$ are isomorphic to intersection of hypersurfaces of type $(2,3,\dots,d-1)$ and $\sum_{i=1}^{d-1} i^r\leq n$, the restriction $\pi_\Sigma:Q_{|\Sigma}\rightarrow \Sigma$ admits a rational section
$\sigma:\Sigma\dashrightarrow Q$. \\

\textit{Case} $1$: The rational section $\sigma$ is actually a rational section of $P_{|\Sigma}\rightarrow \Sigma$.
 This means that for any $x\in \Sigma$, the line $p\circ \sigma(x)$ is contained in $X$. We have the following diagram of resolution of indeterminacies:
$$\xymatrix{\widetilde{\Sigma}\ar[d]^{\tau}\ar[dr]^{\tilde\sigma} & &\\
\Sigma\ar@{.>}[r] &P\ar[r]^-p &F(X)}$$

Let us denote $\mathbb P_{\widetilde{\Sigma}}$ the pull-back via $p\circ \tilde\sigma$ of the $\mathbb{P}^1$-bundle on $F(X)$, $f:\mathbb P_{\widetilde{\Sigma}}\rightarrow X$ the projection on $X$ (which is the restriction of $q$) and $p_\Sigma:\mathbb P_{\widetilde{\Sigma}}\rightarrow \widetilde{\Sigma}$ the projective bundle. The line bundle $\tau^*\mathcal O_{X}(1)_{|\Sigma}$ gives rise to a section $\eta:\widetilde{\Sigma}\rightarrow \mathbb P_{\widetilde{\Sigma}}$ (given by $s\mapsto (p\circ\tilde\sigma(s), \tau(s))$) of $p_\Sigma$. We have the decomposition $\mathrm{Pic}(\mathbb P_{\widetilde{\Sigma}})\simeq \mathbb Z\cdot f^*H_X \oplus p_{\Sigma}^*\mathrm{Pic}(\widetilde{\Sigma})$ so that we can write
\begin{equation}\label{eq_div_case1}
\eta(\widetilde{\Sigma})=f^*H_X+p_\Sigma^*D
\end{equation} for $D$ a divisor on $\widetilde\Sigma$. We apply $f_*$ to that equality: we have $f_*\eta(\widetilde{\Sigma})=\tau_*(\widetilde\Sigma)=\Sigma$ in $\mathrm{CH}_r(X)$. Projection formula yields $f_*f^*H_X=H_X\cdot f_*(1)$. Finally, we see that $f_*p_\Sigma^*D=P_*(p_*\tilde\sigma_*(D))$. So, we get $$\Sigma = H_{X}\cdot f_*(1) + P_*(p_*\tilde\sigma_*(D)).$$ Remembering that $dH_{X}\cdot f_*(1)=i_X^*i_{X,*}f_*(1)\in \mathbb Z\cdot H_{X}^{n-r}$, we are done for this case.\\

\textit{Case} $2$: The rational section $\sigma$ is not a rational section of $P_{|\Sigma}\rightarrow \Sigma$. This means that for the general point $x\in \Sigma$, the line
$\varphi\circ \sigma(x)$ is not contained in $X$, hence intersects
$X$ at $x$ with multiplicity $d$. We have the following diagram of resolution of indeterminacies:
$$\xymatrix{\widetilde{\Sigma}\ar[d]^{\tau}\ar[dr]^{\tilde\sigma} & &\\
\Sigma\ar@{.>}[r] &Q\ar[r]^-\varphi &G(2,n+2)}.$$
Let again $\mathbb P_{\widetilde{\Sigma}}$ be the pull-back via $\varphi\circ \tilde\sigma$ of the $\mathbb{P}^1$-bundle on $G(2,n+2)$
and let $f:\mathbb P_{\widetilde{\Sigma}}\rightarrow \mathbb{P}^{n+1}$ be the natural morphism. Let $\widetilde{\Sigma}_1$ be the locus in $\widetilde{\Sigma}$ consisting of $x\in \widetilde{\Sigma}$ such that
the line $\varphi\circ \tilde\sigma(x)$ is contained in $X$.
We have an equality of $r$-cycles
\begin{eqnarray}
\label{eqncycle} (f_*\mathbb P_{\widetilde{\Sigma}})_{\mid X}=d\Sigma+R
\end{eqnarray} in
${\rm CH}_r(X)$, where the residual cycle $R$ is supported on the $r$-dimensional
locus $\mathbb P_{\widetilde{\Sigma}_1}$, or rather its image in $X$.
It is clear that $R$ is a cycle in the image of $P_*$ so that (\ref{eqncycle}) proves the result in this case.
\end{proof}
\textit{}\\
\indent In the case of smooth cubic hypersurfaces of dimension $\geq 3$, $F(X)$ is always smooth and connected (\cite[Corollary 1.12, Theorem 1.16]{Alt-Kl}). We have the following result which is essentially Theorem \ref{thm_intro1} of the introduction:
\begin{theoreme}\label{prop_gen_5} Let $X\subset \mathbb P^{n+1}_{k}$, with $n\geq 3$ and $char(k)>2$, be a smooth cubic hypersurface containing a linear space of dimension $d\geq 1$. Then, for $1\leq i\leq d$ and $2i\neq n$, $$P_*:\mathrm{CH}_{i-1}(F(X))\rightarrow \mathrm{CH}_i(X)/ \mathbb Z\cdot H_X^{n-i}$$ is surjective on $2\mathrm{CH}_i(X)/ \mathbb Z\cdot H_X^{n-i}$.\\
\indent If moreover, $n\geq 2^r+1$ for some $r>0$ and resolution of singularities holds of $k$-varieties of dimension $r$, then $P_*:\mathrm{CH}_{i-1}(F(X))\rightarrow \mathrm{CH}_i(X)/ \mathbb Z\cdot H_X^{n-i}$ is surjective for any $i\neq \frac{n}{2},\ 1\leq i\leq r$.
\end{theoreme}
\begin{proof} According to \cite[Appendix B]{Cl-Gr}, $X$ admits a unirational parametrization of degree $2$ constructed as follows: for a general line $\Delta$ in $X$, consider the projective bundle $P(T_{X|\Delta})$ over $\Delta$ and the rational map $\varphi: P(T_{X|\Delta})\dashrightarrow X$ which to a point $x\in \Delta$ and a nonzero vector $v\in T_{X,x}$ associates the residual point to $x$ ($x$ has multiplicity $2$) in the intersection $X\cap l_{(x,v)}$ of $X$ with the line of $\mathbb P^{n+1}$ determined by $(x,v)$. The indeterminacy locus $Z$ corresponds to the $(x,v)$ such that $l_{(x,v)}\subset X$. It has codimension $2$ for general lines. Indeed, if $\Delta$ is general, it is generally contained in the locus where the fibers of the projection $q:P\rightarrow X$ are complete (since $P$ has dimension $2n-d$) intersection of type $(2,3)$ in the projectivized tangent spaces so that the general fiber of $Z\rightarrow \Delta$ has dimension $n-3$. Choosing a sufficiently general $\Delta$, we can also assume that $Z$ is smooth. Then, blowing-up $P(T_{X|\Delta})$ along $Z$ yields the resolution of indeterminacies; let us denote $\tau:\widetilde{P(T_{X|\Delta})}\rightarrow P(T_{X|\Delta})$ that blow-up, $E$ the exceptional divisor and $\tilde{\varphi}:\widetilde{P(T_{X|\Delta})}\rightarrow X$ the resulting degree $2$ morphism. For $1\leq i\leq d$, by the formulas for blowing-up, we have the decomposition
$$\mathrm{CH}_i(\widetilde{P(T_{X|\Delta})})=\tau^*\mathrm{CH}_i(P(T_{X|\Delta})) \oplus j_{E,*}\tau_{|E}^*\mathrm{CH}_{i-1}(Z)\oplus j_{E,*}(j_E^*\tilde{\varphi}^*H_{X})\cdot \tau_{|E}^*\mathrm{CH}_i(Z).$$
\indent As $\tau_{|E}$ is flat, we can see that $\tilde{\varphi}_*j_{E,*}\tau_{|E}^*(\cdot)=\tilde{\varphi}_*j_{E,*} [\tau_{|E}^{-1}(\cdot)]$ identifies with the composition of the morphism $\mathrm{CH}_*(Z)\rightarrow \mathrm{CH}_*(F(X))$ (induced by the restriction of natural morphism $P(T_{X})\rightarrow G(2,n+2)$) followed by the action $P_*$.\\
\indent So let $\Gamma\in \mathrm{CH}_i(X)$, with $2i\neq n$, be a cycle on $X$. As $X$ contains a linear space of dimension $i$ and $H_{\textit{\'et}}^{2(n-i)}(X,\mathbb Z_\ell)=\mathbb Z_\ell$ ($\forall \ell\neq char(k)$) by Lefschetz hyperplane section theorem ($n\neq 2i$), for any $\mathcal P\simeq\mathbb P^i\subset X$, $\Gamma -deg(\Gamma)[\mathcal P]$ is homologically trivial and $\tilde{\varphi}_*\tilde{\varphi}^*(\Gamma -deg(\Gamma)[\mathcal P])=2(\Gamma -deg(\Gamma)[\mathcal P])$. As $P(T_{X|\Delta})$ is a projective bundle over $\mathbb P^1$, $\mathrm{CH}_*(P(T_{X|\Delta}))_{hom}=0$ so, from the above discussion, we conclude that there are a $(i-1)$-cycle $\gamma\in \mathrm{CH}_{i-1}(F(X))_{hom}$ and a $i$-cycle $D_\Gamma\in \mathrm{CH}_i(F(X))_{hom}$ such that $$2(\Gamma -deg(\Gamma)[\mathcal P])= P_*\gamma + H_X\cdot P_*D_\Gamma.$$
It remains to deal with the term $H_X\cdot P_*D_\Gamma$.
 For this, let $j:Y\hookrightarrow X$ be a hyperplane section with one ordinary double point $p_0$ as singularity. Then $H_X\cdot P_*D_\Gamma=j_*j^*P_*D_\Gamma$. \\
\indent We have $Y\subset \mathbb P^n$ and if we choose coordinates in which $p_0=[0:\cdots:0:1]$, the equation of $Y$ has the following form: $F(X_0,\cdots,X_n)=X_nQ(X_0,\cdots,X_{n-1})+T(X_0,\cdots,X_{n-1})$ where $Q(X_0,\cdots,X_{n-1})$ is a quadratic homogeneous polynomial and $T(X_0,\cdots,X_{n-1})$ is a degree $3$ homogeneous polynomial. The linear projection $\mathbb P^n\dashrightarrow \mathbb P^{n-1}$ centered at $p_0$ induces a birational map $Y\dashrightarrow \mathbb P^{n-1}\simeq [p_0]$ where $[p_0]$ denotes the scheme parametrizing lines of $\mathbb P^n$ passing through $p_0$. The indeterminacies of the inverse map $\mathbb P^{n-1}\dashrightarrow Y$ are resolved by blowing-up $\mathbb P^{n-1}$ along the complete intersection $F_{p_0}(Y)=\{Q=0\}\cap \{T=0\}$ of type $(2,3)$. The variety of lines of $Y$ passing through $p_0$ is isomorphic to $F_{p_0}(Y)$ and we have the following diagram:
$$\xymatrix{\indent \indent \widetilde{\mathbb P^{n-1}}^{F_{p_0}(Y)}\ar[d]^{\chi}\ar[dr]^{q} & \\
\mathbb P^{n-1}\ar@{.>}[r] &Y}$$
By projection formula, $(j\circ q)_*(j\circ q)^*P_*D_\Gamma= P_*D_\Gamma\cdot j_*q_*1= P_*D_\Gamma\cdot [Y]= P_*D_\Gamma\cdot H_X$ and $(j\circ q)^*P_*D_\Gamma$ is a homologically trivial cycle on $\widetilde{\mathbb P^{n-1}}^{F_{p_0}(Y)}$. But since the ideal $\mathrm{CH}_*(\mathbb P^{n-1})_{hom}$ of homologically trivial cycles on $\mathbb P^{n-1}$ is $0$, from the decomposition of the Chow groups of a blow-up, we get that $(j\circ q)^*P_*D_\Gamma$ can be written $j_{E_{F_{p_0}(Y)},*}\chi_{|E_{F_{p_0}(Y)}}^*w$ for a cycle $w$ on $F_{p_0}(Y)$ so that $H_X\cdot P_*D_\Gamma = j_*q_*j_{E_{F_{p_0}(Y)},*}\chi_{|E_{F_{p_0}(Y)}}^*w$ which can be written $P_*i_{F_{p_0}(Y),*}w$ where $i_{F_{p_0}(Y)}:F_{p_0}(Y)\hookrightarrow F(X)$ is the inclusion. 
Finally, $\mathcal P$ is in $Im(P_*)$ so we have: $2\Gamma = 2\mathcal P + P_*(\gamma + i_{F_{p_0}(Y),*}w)$ which proves that $2\mathrm{CH}_i(X)$ is in the image of $P_*$.\\
\indent When $n\geq 2^r+1$, we can also apply Theorem \ref{thm_form_osc}; we get, for any cycle $\Gamma\in \mathrm{CH}_i(X)$, a cycle $\gamma'\in \mathrm{CH}_{i-1}(F(X))$ such that $3\Gamma + P_*\gamma'\in \mathbb Z\cdot H_X^{n-i}$ in $\mathrm{CH}_i(X)$ so that, putting the two steps together, we get $(3-2)\Gamma + P_*(\gamma' -\gamma - i_{F_{p_0}(Y),*}w) \in \mathbb Z\cdot H_X^{n-i}$ in $\mathrm{CH}_i(X)$.

\end{proof}

\begin{proposition}\label{prop_ruled_hyperpl} Let $X\subset \mathbb P^{n+1}_{k}$, with $n\geq 4$ and $char(k)>2$, be a smooth cubic hypersurface. Then $H_X^{n-2}\in Im(P_*:\mathrm{CH}_1(F(X))\rightarrow \mathrm{CH}_2(X))$. In particular, by Theorem \ref{prop_gen_5}, for $n\geq 5$, $P_*:\mathrm{CH}_1(F(X))\rightarrow \mathrm{CH}_2(X)$ is surjective.
\end{proposition}
\begin{proof} Since, according to \cite[Lemma 1.4]{Mur_alg_mod_rat}, any smooth cubic threefold contains some lines of second type (lines whose normal bundle contains a copy of $\mathcal O_{\mathbb P^1}(-1)$), $X$ contains lines of second type. Let $l_0\subset X$ be a line of second type. According to \cite[Lemma 6.7]{Cl-Gr}, there is a (unique) $\mathbb P^{n-1}\subset \mathbb P^{n+1}$ tangent to $X$ along $l_0$. So, when $n\geq 4$, we can choose a $P_0\simeq \mathbb P^3\subset \mathbb P^{n+1}$ tangent to $X$ along $l_0$. Then $S:=P_0\cap X$ is a cubic surface singular along $l_0$ which is ruled by lines of $X$. Indeed, for any $x\in S\backslash l_0$, $span(x,l_0)\cap S$ is a plane cubic containing $l_0$ with multiplicity $2$; so that the residual curve is a line passing through $x$. So, we can write $S=q(p^{-1}(D))$ for a closed subscheme of pure dimension $1$, $D\subset F(X)$ so, in $\mathrm{CH}_2(X)$, we have $H_X^{n-2}=[S]=P_*([D])$.
\end{proof}

Here is one consequence of this proposition:
\begin{corollaire}\label{cor_spec} Let $\pi:\mathcal{X}\rightarrow B$ be a family of complex cubic hypersurfaces of dimension $n\geq 5$ i.e. $\pi$ is a smooth projective morphism of connected quasi-projective complex varieties with $n$-dimensional cubic hypersurfaces as fibers. Then, the specialization map $$\mathrm{CH}_2(X_{\overline{\eta}})/alg\rightarrow \mathrm{CH}_2(X_t)/alg$$ where $X_{\overline{\eta}}$ is the geometric generic fiber and $X_t:=\pi^{-1}(t)$ for $t\in B(k)$ any closed point, is surjective.
\end{corollaire}
\begin{proof} The statement follows from Proposition \ref{prop_ruled_hyperpl} and the following property, essentially written in the proof of \cite[Lemma 2.1]{Vois_rem_rat_curv}
\begin{proposition}\textit{(\cite[Lemma 2.1]{Vois_rem_rat_curv}).} Let $\pi:\mathcal Y\rightarrow B$ be a smooth projective morphism with rationally connected fibers. Then for any $t\in B(k)$, the specialization map $\mathrm{CH}_1(Y_{\overline{\eta}})/alg\rightarrow \mathrm{CH}_1(Y_t)/alg$ is surjective.
\end{proposition}
\begin{proof} We just recall briefly the proof: by attaching sufficiently very free rational curves to it (so that the resulting curve is smoothable), any curve $C\subset Y_t$ is algebraically equivalent to a (non effective) sum of curves $C_i\subset Y_t$ such that $H^1(C_i, N_{C_i/Y_t})=0$. Then the morphism of deformation of each $(C_i,Y_t)$ to $B$ is smooth. So we have a curve $C_{i,\eta}\subset Y_{K_i}$ where $K_i$ is a finite extension of the function field of $B$, which is sent by specialization in the fiber $Y_t$, to $C_i$.
\end{proof}
Applying this proposition to the relative variety of lines
$F(\mathcal X)\rightarrow B$, yields a surjective map: $\mathrm{CH}_1(F(X_{\overline\eta}))/alg\rightarrow \mathrm{CH}_1(F(X_t))/alg$. The universal $\mathbb P^1$-bundle $\mathcal P\subset F(\mathcal X)\times_B \mathcal X$ gives the surjective maps $\mathcal P_{t,*}:\mathrm{CH}_1(F(X_t))/alg\rightarrow \mathrm{CH}_2(X_t)/alg$ and $\mathcal P_{\overline\eta,*}:\mathrm{CH}_1(F(X_{\overline\eta}))/alg\rightarrow \mathrm{CH}_2(X_{\overline\eta})/alg$ and they commute (\cite[20.3]{Fulton}).
\end{proof}
\textit{}\\

\section{One-cycles on the variety of lines of a Fano hypersurface in $\mathbb P^n$}\label{sec_fano_var}
Throughout this section, $k$ will designate an algebraically closed field. According to \cite[Theorem 4.3, Chap. V]{rat-cur-kol}, for a general hypersurface $X\subset \mathbb P^{n+1}$ of degree $d\leq 2n-2$, the variety of lines $F(X)$ is smooth, connected of dimension $2n-d-1$. In the case of cubic hypersurfaces of dimension $n\geq 3$, we even know, by work of Altman and Kleiman (\cite[Corollary 1.12, Theorem 1.16]{Alt-Kl}, see also \cite{Barth-VdV}) that for any smooth hypersurface $X$, $F(X)$ is smooth and connected.\\
\indent We recall that, for a smooth hypersurface $X\subset \mathbb P_k^{n+1}$ of  degree $d$, when $F(X)$ has the expected dimension $2n-d-1$, it is the zero-locus in $G(2,n+2)$ of a regular section of ${\rm Sym}^d(\mathcal E_2)$, where $\mathcal E_2$ is the rank $2$ quotient bundle on $G(2,n+2)$ and its dualizing sheaf, given by adjunction formula (\cite[Theorem III 7.11]{Hart_alg_geo}), is $-((n+2) - \frac{d(d+1)}{2}))$ times the Pl\"ucker line bundle on $G(2,n+2)$ restricted to $F(X)$. In particular, when $F(X)$ is smooth, connected and $\frac{d(d+1)}{2}<(n+2)$, $F(X)$ is Fano so rationally connected.\\
\indent From now, we assume that the condition $d(d+1)< 2(n+2)$ holds and that $X\subset \mathbb P^{n+1}_k$ is a smooth hypersurface such that $F(X)$ is smooth and connected. Then the following theorem applies to $F(X)$ if $char(k)=0$ or, when $char(k)>0$, if $F(X)$ is, moreover separably rationally connected:
\begin{theoreme}\label{thm1_tian_zong}\textit{(\cite[Theorem 1.3]{Tian-Zong}).} Let $Y$ be a smooth proper and separably rationally connected variety over an algebraically closed field. Then every $1$-cycle is rationally equivalent to a $\mathbb Z$-linear combination of cycle classes of rational curves. That is, the Chow group $\mathrm{CH}_1(Y)$ is generated by rational curves.
\end{theoreme}

\begin{corollaire}\label{cor_csq1} When $char(k)=0$ and $X$ is a smooth cubic hypersurface of dimension $\geq 5$, $F(X)$ is separably rationally connected; then Proposition \ref{prop_gen_5} together with Theorem \ref{thm1_tian_zong} yields that $\mathrm{CH}_2(X)$ is generated by classes of rational surfaces. In positive characteristic, the same is true for smooth cubic hypersurfaces $X$ whose variety of lines $F(X)$ is separably rationally connected.
\end{corollaire}
\begin{remarque} \normalfont When $k=\mathbb C$ and $X$ is a smooth cubic hypersurface of dimension $5$, the group of $1$-cycles modulo algeraic equivalence, $\mathrm{CH}_1(F(X))/alg$ is finitely generated. Indeed, according to \cite[Theorem 5.7, Chap. II]{rat-cur-kol}, any rational curve is algebraically equivalent to a sum of rational curves of anticanonical degree at most $dim_k(F(X))+1$. As there are finitely many irreducible varieties parametrizing rational curves of bounded degree, $\mathrm{CH}_1(F(X))/alg$, is finitely generated. So, by the surjectivity of $P_*$, $\mathrm{CH}_2(X)/alg$ is finitely generated. So $H_{nr}^4(X,\mathbb Q/\mathbb Z)\simeq \mathrm{CH}_2(X)_{tors, AJ}/alg\subset \mathrm{CH}_2(X)/alg$ is finitely generated and being a functorial birational invariant of a cubic hypersurface, $2$-torsion. So by this geometric method, we are just able to prove the finiteness of the group $H_{nr}^4(X,\mathbb Q/\mathbb Z)$. By more algebraic methods, Kahn and Sujatha (\cite{KS}) prove the vanishing of that group.
\end{remarque}
Actually, by completely different methods using a variant of \cite[Theorem 18]{vois_4fold_ihc}, the first item of Corollary \ref{cor_csq1} turns out to be true for cubic $4$-folds also in characteristic $0$.
\begin{proposition}\label{provoisin} Let $X\subset \mathbb P^5_\mathbb C$ be a smooth cubic hypersurface. Then $\mathrm{CH}_2(X)$ is generated by classes of rational surfaces.
\end{proposition}
\begin{proof} In the  proof by Voisin of the integral Hodge conjecture of cubic $4$-folds (\cite[Theorem 18]{vois_4fold_ihc}), one can replace the parametrization of the family of intermediate jacobians associated to a Lefschetz pencil of $X$, with rationally connected fibers given by \cite{Mar-Tik} and \cite{Iliev-Mar} the family of elliptic curves
of degree $5$) by the one given by \cite[Theorem 9.2]{HRS} (the family of rational curves of degree $4$); her proof then shows that any degree $4$ Hodge class is homologically equivalent to the class of a combination of rational surfaces
swept-out by
a family of rational curves of degree $4$ in $X$
parameterized by a rational curve. Finally, since $X$ is rationally connected and the intermediate jacobian $J^3(X)$ is trivial,  Bloch-Srinivas \cite[Theorem 1]{Bl-Sr} applies and says that codimension $2$ cycles homologically trivial
 on $X$ are rationally trivial so that  we have proved that any $2$-cycle on $X$ is rationally equivalent to a combination of rational surfaces.
\end{proof}

\subsection{One-cycles modulo algebraic equivalence} In this section, we apply the methods
 of \cite[Theorem 6.2]{Tian-Zong}, using a coarse parametrization of rational curves lying on $F(X)$, to study
 $1$-cycles on varieties $F(X)$. Our goal is to prove:

\begin{theoreme}\label{thm_lines_alg_eq} Let $X\subset \mathbb P^{n+1}_k$ be a smooth hypersurface of degree $d$ over an algebraically closed field of characteristic $0$, with $\frac{d(d+1)}{2}<n$, such that $F(X)$ is smooth, connected. Then every rational curve on $F(X)$ is algebraically equivalent to an integral sum of lines. In particular, 
any $1$-cycle on $F(X)$ is algebraically equivalent to an integral sum of lines and thus ${\rm CH}_1(F(X))_{hom}={\rm CH}_1(F(X))_{alg}$.
\end{theoreme}

 We start with some preparation. Let $V$ be a $(n+2)$-dimensional $k$-vector space and $X\subset \mathbb P(V)\simeq \mathbb P^{n+1}_k$ a smooth hypersurface of degree $d$. A morphism $r:\mathbb P^1\rightarrow G(2,V)$ such that $r^*\mathcal O_{G(2,V)}(1)\simeq \mathcal O_{\mathbb P^1}(e)$, with $e\geq 1$, is associated to the datum of a globally generated rank $2$ vector bundle on $\mathbb P^1$, which is a quotient of the trivial bundle $V\otimes \mathcal O_{\mathbb P^1}$ i.e. to an exact sequence $$V\otimes \mathcal O_{\mathbb P^1}\rightarrow \mathcal O_{\mathbb P^1}(a)\oplus\mathcal O_{\mathbb P^1}(b)\rightarrow 0$$ with $a,b\geq 0$ and $a+b=e$. So a natural parameter space for those morphisms is $$\mathbb P:=\mathbb P(Hom(V^*,H^0(\mathbb P^1, \mathcal O_{\mathbb P^1}(a))\oplus H^0(\mathbb P^1,\mathcal O_{\mathbb P^1}(b)))).$$
Given $[P_0,\dots,P_{n+1},Q_0,\dots,Q_{n+1}]\in \mathbb P$, where the $P_i$'s are in $H^0(\mathbb P^1, \mathcal O_{\mathbb P^1}(a))$ and the $Q_i$'s are in $H^0(\mathbb P^1,\mathcal O_{\mathbb P^1}(b))$, the points in the image in $\mathbb P(V)$ of $Im(\mathbb P^1\rightarrow G(2,n+2))$ under the correspondence given by the universal $\mathbb P^1$-bundle are of the form $[P_0(Y_0,Y_1)\lambda+Q_0(Y_0,Y_1)\mu,\dots, P_{n+1}(Y_0,Y_1)\lambda+Q_{n+1}(Y_0,Y_1)\mu]$ where ${\rm Span}(Y_0,Y_1)=H^0(\mathbb P^1,\mathcal O_{\mathbb P^1}(1))$. Let $\Pi_{i=0}^{n+1} X_i^{\alpha_i}\in H^0(\mathbb P^{n+1},\mathcal O_{\mathbb P^{n+1}}(d))$ be a mononial with $\sum_{i=0}^{n+1}\alpha_i=d$. Then the induced equation on the image in $\mathbb P^{n+1}$ of the morphism $\mathbb P^1\rightarrow G(2,n+2)$ associated to $[P_0,\dots, P_{n+1},Q_0,\dots,Q_{n+1}]$ has the following form: $$\sum_{k=0}^d (\sum_{\substack{0\leq l_0\leq \alpha_0,\dots, 0\leq l_{n+1}\leq \alpha_{n+1} \\ \sum_i l_i=k}}\Pi_{i=0}^{n+1}\binom{\alpha_i}{l_i}P_i^{\alpha_i-l_i}Q_i^{l_i}) \lambda^{d-k}\mu^k$$ so that, denoting $\mathbb P_X$, the closed subset of $\mathbb P$ parametrizing the $[P_0,\dots, P_{n+1},Q_0,\dots,Q_{n+1}]$ whose image in $\mathbb P^{n+1}$ is contained in the hypersurface $X$ of degree $d$, $\mathbb P_X$ is defined by $\sum_{k=0}^d (a(d-k)+bk+1)$ homogeneous polynomials of degree $d$ on $\mathbb P$.\\
\indent The closed subset $B\subset \mathbb P$ parametrizing the $M\in \mathbb P(Hom(V^*,H^0(\mathbb P^1, \mathcal O_{\mathbb P^1}(a))\oplus H^0(\mathbb P^1,\mathcal O_{\mathbb P^1}(b))))$ whose rank is $\leq 2$ has dimension $2(e+n)+3$. Now, we have the following lemma:
\begin{lemme}\label{lem_hart_connected}\textit{(\cite{Hart_conn}).} Let $Y$ be a subscheme of a projective space $\mathbb P^N$ defined by $M$ homogeneous polynomials. Let $Z$ be a closed subset of $Y$ with dimension $< N-M-1$. Then $Y\backslash Z$ is connected.
\end{lemme}
The closed subset $B\cap \mathbb P_X$ of $\mathbb P_X$ has dimension $$dim_k(F(X))+2(a+1)+2(b+1)-1= 2n-d-1+2e+3=2(e+n)-d+2$$ since it parametrizes (generically) a point of $F(X)$ and over that point $2$ polynomials in $H^0(\mathbb P^1,\mathcal O_{\mathbb P^1}(a))$ and $2$ polynomials in $H^0(\mathbb P^1,\mathcal O_{\mathbb P^1}(b))$. Applying Lemma \ref{lem_hart_connected} with $Y=\mathbb P_X$ and $Z=B\cap \mathbb P_X$, so that $N=(n+2)(e+2)-1$ and $M=\sum_{k=0}^d (a(d-k)+bk+1)=d+1+e\frac{d(d+1)}{2}$, yields the following condition for the connectedness of $\mathbb P_X\backslash (B\cap \mathbb P_X)$: \begin{equation}\label{ineq_connected}
e(n-\frac{d(d+1)}{2})>1
\end{equation}

\begin{proof}[Proof of Theorem \ref{thm_lines_alg_eq}] We proceed by induction on the degree of the considered rational curve, following the arguments of \cite[Theorem 6.2]{Tian-Zong}.\\
\indent Let $D\subset \mathbb P$ be the closed subset parametrizing $2(n+2)$-tuples $[P_0,\dots,P_{n+1},Q_0,\dots,Q_{n+1}]$ that have a common non constant factor. Assume $e\geq 2$. Let $p\in \mathbb P_X\backslash(\mathbb P_X\cap(B\cup D))$ be a point parametrizing a degree $e$ morphism $\mathbb P^1\rightarrow F(X)$ generically injective. As $e\geq 2$, $\mathbb P_X\backslash(\mathbb P_X\cap B)$ is connected; so there is a connected curve $\gamma$ in $\mathbb P_X\backslash(\mathbb P_X\cap B)$ connecting $p$ to a point $q=[P_{0,q},\dots,P_{n+1,q},Q_{0,q},\dots,Q_{n+1,q}]$ of $\mathbb P_X\cap D\backslash (\mathbb P_X\cap B)$. Factorizing out the common factor of $(P_{i,q}, Q_{i,q})_{i=0\dots n+1}$, we get a $(P'_{i,q}, Q'_{i,q})_{i=0\dots n+1}$ which parametrizes a morphism $\mathbb P^1\rightarrow F(X)$ of degree $<e$ (finite onto its image), since $q\notin B$. So, approching $q$ from points of $\gamma$ outside $D$ and using standard bend-and-break construction, we get from $q$ a morphism from a connected curve whose components are isomorphic to $\mathbb P^1$ to $F(X)$ such that the restriction to each component yields a rational curve of degree $<e$ (or a contraction). So the rational curve parametrized by $p$ is algebraically equivalent to a sum of rational curve each of which has degree $<e$. We conclude by induction on $e$ that the rational curve parametrized by $p$ is algebraically equivalent to a sum of lines.
\end{proof}

\subsection{One-cycles modulo rational equivalence} From now on, we will assume that $X\subset \mathbb P^{n+1}_k$ is a smooth hypersurface of degree $d>2$, with $d(d+1)/2 < n$, and that $char(k)=0$. The following is proved in
\cite[Proposition 6.2]{Deb-Man}:

\begin{proposition}\label{prop_chain_conn_lines} Assume $char(k)=0$ and $X\subset \mathbb P^{n+1}_k$ is a smooth hypersurface of degree $d>2$, with $d(d+1)/2 < n$. Then, $F(X)$ is chain connected by lines.
\end{proposition}
Proceeding as in \cite{Tian-Zong}, we get the following result:
\begin{theoreme}\label{thm_gen_lines} Let $X\subset \mathbb P^{n+1}_k$ be a smooth hypersurface of degree $d>2$ over an algebraically closed field of characteristic $0$, with $\frac{d(d+1)}{2}< n$, such that $F(X)$ is smooth and connected. Then $\mathrm{CH}_1(F(X))$ is generated by lines i.e. any $1$-cycle is rationally equivalent to a $\mathbb Z$-linear combination of lines.
\end{theoreme}
\begin{proof} Let $\gamma$ be a $1$-cycle on $F(X)$. According to Theorem \ref{thm_lines_alg_eq}, there is a $\mathbb Z$-linear combination of lines $\sum_i m_il_i$ such that $\gamma-\sum_im_il_i$ is algebraically equivalent to $0$. Then, using \cite[Proposition 3.1]{Tian-Zong} and Proposition \ref{prop_chain_conn_lines} (via \cite[Lemma IV 3.4 and Proposition IV 3.13.3]{rat-cur-kol}), we know there is a positive integer $N$ such that for every $1$-cycle $C$ on $F(X)$, $N C$ is rationally equivalent to a $\mathbb Z$-linear combination of lines. As the group
$\mathrm{CH}_1(X)_{alg}$ of $1$-cycles algebraically equivalent to $0$  is divisible (\cite[Lemme 0.1.1]{Beau_prym}),  we conclude that $\gamma-\sum_im_il_i$ is rationally equivalent to a $\mathbb Z$-linear combination of lines.
\end{proof}

This provides us  the following results for cubic hypersurfaces (cf. Proposition \ref{prop_intro2}):

\begin{corollaire}\label{cor_surf_cubic} Let $k$ be an algebraically closed field of characteristic $0$ and $X$ a smooth cubic hypersurface. We have the following properties:\\
\indent (i) if $dim_k(X)\geq 7$, then, $\mathrm{CH}_2(X)$ is generated (over $\mathbb Z$) by cycle classes of planes contained in $X$ and ${\rm CH}_2(X)_{hom}={\rm CH}_2(X)_{alg}$;\\
\indent (ii) if $dim_k(X)\geq 9$, then, $\mathrm{CH}_2(X)\simeq \mathbb Z$
\end{corollaire}
\begin{proof} Item (i) is an application of Proposition \ref{prop_ruled_hyperpl} and Theorem \ref{thm_gen_lines}.\\
\indent (ii) The variety of lines $F(F(X))$ of $F(X)$ is isomorphic to the projective bundle $\mathbb P(\mathcal E_{3|F_2(X)})$ over $F_2(X)\subset G(3,n+2)$, where $\mathcal E_3$ is the rank $3$ quotient bundle on $G(3,n+2)$ and $F_2(X)$ is the variety on planes of $X$, since a line in $F(X)$ correspond to the lines of $\mathbb P^{n+1}$ contained in a plane $\mathbb P^2\simeq P\subset \mathbb P^{n+1}$ passing through a given point of $P$. Now, when $n\geq 9$, according to \cite[Proposition 6.1]{Deb-Man}, $\mathrm{CH}_0(F_2(X))\simeq\mathbb Z$ so that $\mathrm{CH}_0(F(F(X)))\simeq \mathbb Z$.
\end{proof}

\section{Inversion formula}
Let $X\subset \mathbb P^{n+1}_k$, where $n\geq 3$, be a smooth cubic hypersurface over a field $k$. Let as before $F(X)\subset G(2,n+2)$ be the variety of lines of $X$ and $P\subset F(X)\times X$ the correspondence given by the universal $\mathbb P^1$-bundle over $F(X)$. The variety $F(X)$ is smooth, connected of dimension  $2n-4$ (\cite[Corollary 1.12, Theorem 1.16]{Alt-Kl}). 
\subsection{Inversion formula}\label{sec_single}In this section, adapting constructions and arguments developped in \cite{Shen_univ} (see also \cite{Shen_main}), we establish an inversion formula for a smooth subvariety $\Sigma$ of $X$.\\
\indent For subvarieties $\Sigma$ in general position, this formula express the class of $\Sigma$ in $\mathrm{CH}_{dim(\Sigma)}(X)$ in terms of the class of the subscheme of $F(X)$ consisting of the lines of $X$ bisecant to $\Sigma$.\\
\indent First of all, the lines of $\mathbb P^{n+1}_k$ bisecant to any subvariety $\Sigma$ are naturally in relation with the punctual Hilbert scheme $Hilb_2(\Sigma)$, that we shall denote $\Sigma^{[2]}$, via the morphism \begin{equation}\label{morph_grass} \varphi:\Sigma^{[2]}\rightarrow G(2,n+2)\end{equation}
 which associates to a length $2$ subscheme of $\Sigma$ the line it determines.\\
\indent We recall that for any smooth variety $Y$, $Y^{[2]}$ is smooth and is obtained as the quotient of the blow-up $\widetilde{Y\times Y}$ of $Y\times Y$ along the diagonal $\Delta_Y$, by its natural involution. Let us denote $q:\widetilde{Y\times Y}\rightarrow Y^{[2]}$ the quotient morphism, $\tau:\widetilde{Y\times Y}\rightarrow Y\times Y$ the blow-up and $j_{E_Y}:E_Y\hookrightarrow \widetilde{Y\times Y}$ the exceptional divisor of $\tau$. As the involution acts trivially on $E_Y$, $q_{|E_Y}$ is an isomorphism onto its image and $q$ is a double cover of $Y^{[2]}$ ramified along $q(E_Y)$. So let us denote $\delta_Y\in \mathrm{CH}^1(Y^{[2]})$ a divisor satisfying $[q(E_Y)]=2\delta_Y$.\\
\indent For a subvariety $\Sigma$ of $X$ in general position, the relation between lines of $X$ bisecant to $\Sigma$ and $\Sigma$ rests on the existence of a residual map:
\begin{equation}\label{resid_map} r:\Sigma^{[2]}\dashrightarrow X
\end{equation}
associating to a length $2$ subscheme of $\Sigma$, $x+y$, the point $z\in X$ residual to $x+y$ in the intersection of $l_{(x+y)}\cap X$, $l_{(x+y)}$ being the line determined by $x+y$. The map (\ref{resid_map}) is not defined on length $2$ subschemes whose associated line is contained in $X$.\\
\indent Let us denote $P_2$ the subscheme of $X^{[2]}$ of length $2$ subschemes of $X$, whose associated line is contained in $X$ and let us denote $i_{P_2}:P_2\hookrightarrow X^{[2]}$ the embedding. We can see that $P_2$ admits a structure $p_{P_2}:P_2\rightarrow F(X)$ of $\mathbb P^2$-bundle over $F(X)$ as $P_2$ is the symmetric product of $P$ over $F(X)$. In particular $P_2$ is a smooth subvariety of $X^{[2]}$ of codimension $2$.\\
\indent Now, for any smooth subvariety $\Sigma\subset X$, we prove the following inversion formula:

\begin{theoreme}\label{thm_form_single} Let $X\subset \mathbb P^{n+1}_k$ be a smooth cubic hypersurface and $\Sigma\subset X$ a smooth subvariety of dimension $1\leq d\leq n$. Then, the following equality holds in $\mathrm{CH}_d(X)$:
\begin{equation}\label{eq_inv_form_thm}
(2deg(\Sigma)-3)[\Sigma] + P_*[(p_{P_2,*}i_{P_2}^*\Sigma^{[2]})\cdot c_1(\mathcal O_{F(X)}(1))^{d-1}] = m(\Sigma)H_X^{n-d}
\end{equation}
where $m(\Sigma)$ is an integer, $H_X=c_1(\mathcal O_{X}(1))$ and $\mathcal O_{F(X)}(1)$ is the Pl\"ucker line bundle on $F(X)$.
\end{theoreme}

Let us start with an analysis of the geometry of (\ref{resid_map}) for $X$. The indeterminacies of $$r:X^{[2]}\dashrightarrow X$$ are resolved by blowing up $X^{[2]}$ along $P_2$. Let us denote $\chi:\widetilde{X^{[2]}}\rightarrow X^{[2]}$ this blow-up morphism and $E_{P_2}$ the exceptional divisor. The variety $\widetilde{X^{[2]}}$ is naturally a subvariety of $X^{[2]}\times X$ and, as such, can be regarded also as a correspondence between $X^{[2]}$ and $X$. In view of the relation between the bisectant lines of a subvariety $\Sigma\subset X$ in general position, which as to do with $\Sigma^{[2]}$, and $\Sigma$, we want to be able to compute the action of the correspondence $\widetilde{X^{[2]}}$.\\
\indent We recall that we have a morphism $\varphi: X^{[2]}\rightarrow G(2,n+2)$ (\ref{morph_grass}) from which we get, by pulling back objects from $G(2,n+2)$, a diagram:
$$\xymatrix{\mathbb P(\varphi^*\mathcal E_2)\ar[r]^f\ar[d]_\pi & \mathbb P^{n+1}\\
X^{[2]} & }$$
We have the following proposition:
\begin{proposition}\label{prop_taut_is_res} (i) There is an embedding $\sigma:\widetilde{X\times X}\rightarrow \mathbb P(\varphi^*\mathcal E_2)$ given by $(x,y)\mapsto (l_{(x,y)},x)$ if $(x,y)\in \widetilde{X\times X}\backslash E$ and $(x,v)\mapsto (l_{(x,v)},x)$ if $(x,v)\in E\simeq P(T_X)$.\\
\indent (ii) The class of $\sigma(\widetilde{X\times X})$ in $\mathrm{Pic}(\mathbb P(\varphi^*\mathcal E_2))$ is:
\begin{equation}\label{eq_class_div_taut} \sigma(\widetilde{X\times X})= 2f^*H - \pi^*(q_*\tau^*pr_1^*H_X -2\delta_X)
\end{equation}
where $H=c_1(\mathcal O_{\mathbb P^{n+1}_k}(1))$, $H_X$ is the restriction of $H$ to $X$ and $pr_1:X\times X\rightarrow X$ the first projection.\\
\indent (iii) We have an inclusion of divisors $\sigma(\widetilde{X\times X})\subset f^*(X)$ and the residual divisor to $\sigma(\widetilde{X\times X})$ in $f^*(X)$ is isomorphic to $\widetilde{X^{[2]}}$ and $\pi_{|\widetilde{X^{[2]}}}=\chi$ so that the class of $\widetilde{X^{[2]}}$ in $\mathrm{Pic}(\mathbb P(\varphi^*\mathcal E_2))$ is:
\begin{equation}\label{eq_class_blow_up} \widetilde{X^{[2]}}= f^*H +\pi^*(q_*\tau^*pr_1^*H_X-2\delta_X)
\end{equation}
\end{proposition}
\begin{proof} As for any point $p\in \widetilde{X\times X}$, the point $pr_1(\tau(p))$ lies on the line $\varphi(q(p))$ (defined over $k(p)$), the evaluation morphism $q^*\varphi^*\mathcal E_2\rightarrow \tau^*pr_1^*\mathcal O_X(1)$, where $\mathcal O_X(1)\simeq \mathcal O_{\mathbb P^{n+1}_k}(1)_{|X}$, is surjective i.e. gives rise to a section $\sigma'$ of the projective bundle $\pi':\mathbb P(q^*\varphi^*\mathcal E_2)\rightarrow \widetilde{X\times X}$. Let us denote $q':\mathbb P(q^*\varphi^*\mathcal E_2)\rightarrow \mathbb P(\varphi^*\mathcal E_2)$ the morphism obtained from $q$ by base change; it is also a ramified double cover. The composition $\sigma:=q'\circ \sigma':\widetilde{X\times X}\rightarrow \mathbb P(\varphi^*\mathcal E_2)$ is an isomorphism onto its image and we have the inclusion of divisors $\sigma(\widetilde{X\times X})\subset f^{-1}(X)$. Let us denote $R$ the residual scheme to $\sigma(\widetilde{X\times X})$ in $f^{-1}(X)$. We need to prove that $\pi_{|R}$ is the blow-up of $\Sigma^{[2]}$ along $P_2$ i.e. $R\simeq \widetilde{X^{[2]}}$.\\

\indent As $\sigma'$ is the section of the projective bundle $\mathbb P(q^*\varphi^*\mathcal E_2)$ given by $\tau^*pr_1^*\mathcal O_{X}(1)$, its class in $\mathrm{CH}^1(\mathbb P(q^*\varphi^*\mathcal E_2))$ is given by $c_1(\pi'^*\mathcal K^\vee\otimes\mathcal O_{\mathbb P(q^*\varphi^*\mathcal E_2)}(1))$, where $\mathcal K$ is defined by the exact sequence:
\begin{equation}
0\rightarrow \mathcal K\rightarrow q^*\varphi^*\mathcal E_2\rightarrow \tau^*pr_1^*\mathcal O_{X}(1)\rightarrow 0
\label{ex_seq_def_K}
\end{equation} and $\mathcal O_{\mathbb P(q^*\varphi^*\mathcal E_2)}(1)\simeq q'^*f^*\mathcal O_{\mathbb P^{n+1}_k}(1)$. We have:
$$\begin{tabular}{llll}
$[\sigma(\widetilde{X\times X})]$ &$=$ &$q'_*\sigma'_*(\widetilde{X\times X})$\\
$ $ &$=$ &$q'_*c_1(\pi'^*\mathcal K^\vee\otimes \mathcal O_{\mathbb P(q^*\varphi^*\mathcal E_2)}(1))$\\
$ $ &$=$ &$q'_*\pi'^*c_1(\mathcal K^\vee) + q'_*c_1(q'^*f^*\mathcal O_{\mathbb P^{n+1}}(1))$\\
$ $ &$=$ &$\pi^*q_*c_1(\mathcal K^\vee) + c_1(f^*\mathcal O_{\mathbb P^{n+1}}(1))\cdot q'_*(1)\ \mathrm{since}\ q\ \mathrm{and}\ \pi\ \mathrm{are\ proper\ and\ flat}$\\
$ $ &$=$ &$\pi^*q_*c_1(\mathcal K^\vee) + 2c_1(f^*\mathcal O_{\mathbb P^{n+1}}(1))\ \mathrm{since}\ q'\ \mathrm{is\ a\ double\ cover}$\\
$ $ &$=$ &$\pi^*q_*[c_1(\tau^*pr_1^*\mathcal O_X(1)) - c_1(q^*\varphi^*\mathcal E_2)] + 2f^*H\ \mathrm{using}\ (\ref{ex_seq_def_K})$\\
$ $ &$=$ &$\pi^*[q_*\tau^*pr_1^*c_1(\mathcal O_X(1)) - 2c_1(\varphi^*\mathcal E_2)] + 2f^*H\ \mathrm{since}\ q\ \mathrm{is\ a\ double\ cover}$
\end{tabular}$$


As a linear form on $\mathbb P^1$ is determined by its value on a length $2$ subscheme, the evaluation morphism yields an isomorphism of sheaves:
\begin{equation}\varphi^*\mathcal E_2\simeq q_*\tau^*pr_1^*\mathcal O_{X}(1),
\label{isom_taut}
\end{equation}
so that, using Grothendieck-Riemann-Roch theorem for $q$, we have the equality $c_1(\varphi^*\mathcal E_2)= q_*c_1(\tau^*pr^*_1\mathcal O_{X}(1)) - \delta_X$. We end the computation of $[\sigma(\widetilde{X\times X})]$ as follows:
$$\begin{tabular}{llll}
$[\sigma(\widetilde{X\times X})]$ &$=$ &$2f^*H + \pi^*[q_*\tau^*pr_1^*c_1(\mathcal O_X(1)) - 2c_1(\varphi^*\mathcal E_2)]$\\
$ $ &$=$ &$2f^*H + \pi^*[(1-2)q_*\tau^*pr_1^*c_1(\mathcal O_X(1)) - 2\delta_X]$
\end{tabular}$$


\indent Now, we have $R=[f^{-1}(X)]-[\sigma(\widetilde{X\times X})]=3f^*H-[\sigma(\widetilde{X\times X})]=f^*H+\pi^*(q_*\tau^*pr_1^*H_{X} -2\delta_X)$ so that by projection formula $\pi_*\mathcal O_{\mathbb P(\varphi^*\mathcal E_2)}(R)\simeq \varphi^*\mathcal E_2\otimes \mathcal O_{X^{[2]}}(q_*\tau^*pr_1^*H_{X} -2\delta_X)$. Letting $s_R\in |\mathcal O_{\mathbb P(\varphi^*\mathcal E_2)}(R)|$ be a section whose zero locus is equal to $R$, we can consider $s_R$ as a section of the rank $2$-vector bundle $\pi_*\mathcal O_{\mathbb P(\varphi^*\mathcal E_2)}(R)$. Then the zero locus of this section corresponds to length $2$ subschemes whose associated line is contained in $X$ that is to $P_2$. So the class $P_2$ in $\mathrm{CH}^2(X^{[2]})$ is $c_2(\pi^*\mathcal O_{\mathbb P(\varphi^*\mathcal E_2)}(R))$.\\
\indent Let $U\simeq Spec(A)$ be an affine open subset of $X^{[2]}$ such that $\mathbb P(\varphi^*\mathcal E_2)_{|U}\simeq \mathbb P^1_A$. Denoting $[Y_0:Y_1]$ the homogeneous (relative) coordinates on $\mathbb P^1_A$, the equation $s_{R}$ of $R_{|U}\subset \mathbb P^1_A$, is of the form $f_0Y_0+f_1Y_1=0$, where $f_0,f_1\in A$, since $R\in |\mathcal O_{\mathbb P(\varphi^*\mathcal E_2)}(1)\otimes \pi^*\mathcal O_{X^{[2]}}(q_*\tau^*pr_1^*H_{X} -2\delta_X)|$. Then the section $s_R$ of $(\pi_*\mathcal O_{\mathbb P(\varphi^*\mathcal E_2)}(R))_{|U}$ is $(f_0,f_1)$. As $P_2$ is the zero locus of $s_R$, the ideal of $P_2\cap U$ in $U$ is generated by $(f_0,f_1)$ and as $P_2$ is smooth of codimension $2$, $(f_0, f_1)$ is a regular sequence in $A$. As $(f_0, f_1)$ is a regular sequence, the equation $f_0Y_0+f_1Y_1=0$ tells exactly that $R$ is the blow-up of $X^{[2]}$ along $P_2$ i.e. $R\simeq \widetilde{X^{[2]}}$.
\end{proof}

\indent The divisors $\widetilde{X^{[2]}}$, $\sigma(\widetilde{X\times X})$ and $f^*X=[f^{-1}(X)]$ can be considered as correspondences from $X^{[2]}$ to $X$. The following fiber square:
$$\xymatrix{ f^{-1}(X)\ar@{^{(}->}[d]^{i'_X}\ar[r]^{f'} &X\ar@{^{(}->}[d]^{i_X}\\
 \mathbb P(\varphi^*\mathcal E_2) \ar[d]^{\pi}\ar[r]^f &\mathbb P^{n+1}\\
 X^{[2]} &}$$
yields the following easy lemma:
\begin{lemme}\label{lem_fact_proj_space} The action $[f^{*}(X)]_*:\mathrm{CH}^*(X^{[2]})\rightarrow \mathrm{CH}^*(X)$ factors through $\mathrm{CH}^*(\mathbb P^{n+1})$ i.e. for any $z\in \mathrm{CH}^i(X^{[2]})$, there is an integer $m_z\in \mathbb Z$ such that $[f^{-1}(X)]_*z=m_zH^{n+i-2d}_{X}$.
\end{lemme}
\indent By Lemma \ref{lem_fact_proj_space}, $\left[\widetilde{X^{[2]}}\right]_*+[\sigma(\widetilde{X\times X})]_*:\mathrm{CH}^*(X^{[2]})\rightarrow \mathrm{CH}^*(X)$ factors through $\mathrm{CH}^*(\mathbb P^{n+1}_k)$. As $[\sigma(\widetilde{X\times X})]$ is tautological, we can compute the action of $\left[\widetilde{X^{[2]}}\right]$ modulo cycles coming from $\mathbb P^{n+1}_k$. We now have to find a suitable relation on which we can use the action of $\left[\widetilde{X^{[2]}}\right]$.

\begin{lemme}\label{lem_rel_divv} We have the following equality in $\mathrm{CH}^1(\widetilde{X^{[2]}})$:
\begin{equation}\label{eq_rel_div_fin} (f^*H)_{|\widetilde{X^{[2]}}} = 2\pi_{|\widetilde{X^{[2]}}}^*q_*\tau^*pr_1^*H_{X} - 3\pi_{|\widetilde{X^{[2]}}}^*\delta_X - E_{P_2}.
\end{equation}
\end{lemme}

\begin{proof} Using that $\pi_{|\widetilde{X^{[2]}}}$ is a blow-up, we have $K_{\widetilde{X^{[2]}}}=\pi^*_{|\widetilde{X^{[2]}}}K_{X^{[2]}} + E_{P_2}$. Secondly, adjunction formula gives $K_{\widetilde{X^{[2]}}}=(K_{\mathbb P(\varphi^*\mathcal E_2)}+\widetilde{X^{[2]}})_{|\widetilde{X^{[2]}}}$. As $K_{\mathbb P(\varphi^*\mathcal E_2)}=\pi^*K_{X^{[2]}}+K_{\mathbb P(\varphi^*\mathcal E_2)/X^{[2]}}$, clearing $\pi^*K_{X^{[2]}}$, we get
\begin{equation}\label{eq_div}
E_{P_2}=(K_{\mathbb P(\varphi^*\mathcal E_2)/X^{[2]}}+\widetilde{X^{[2]}})_{|\widetilde{X^{[2]}}}
\end{equation}
\indent Using formulas for projective bundle, we have $$K_{\mathbb P(\varphi^*\mathcal E_2)/X^{[2]}}=-2c_1(\mathcal O_{\mathbb P(\varphi^*\mathcal E_2)}(1)) +\pi^*c_1(\varphi^*\mathcal E_2) = -2f^*H + \pi^*(q_*\tau^*pr^*_1c_1(\mathcal O_{X}(1)) - \delta_X).$$ 
\indent Then, (\ref{eq_div}) yields $$E_{P_2}=(-f^*H + 2\pi^*q_*\tau^*pr^*_1H_{X} -3\pi^*\delta_X)_{|\widetilde{X^{[2]}}}$$
\end{proof}

\begin{proof}[Proof of Theorem \ref{thm_form_single}] Let $i_{\Sigma}:\Sigma\hookrightarrow X$ be a smooth subvariety of $X$ of dimension $d$. Then we have the description of $\Sigma^{[2]}$ as the quotient of the blow-up $\widetilde{\Sigma\times\Sigma}$ of $\Sigma\times \Sigma$ along the diagonal by the involution. The class of $\widetilde{\Sigma\times\Sigma}$, which is the strict transform of $\Sigma\times\Sigma$ under $\tau$, in $\mathrm{CH}_{2d}(\widetilde{X\times X})$ is given by the excess formula (\cite[Theorem 6.7 and Corollary 4.2.1]{Fulton}):
\begin{equation}\label{eq_excess_blow_up} \widetilde{\Sigma\times \Sigma} = \tau^*(\Sigma\times \Sigma) - j_{E_X,*}\{c(\tau^*_{|E_X}T_X)(1+E_{X|E_X})^{-1}\cdot \tau^*_{|E_X}i_{\Sigma,*}(c(T_\Sigma)^{-1})\}_{2d}\ \mathrm{in\ CH}_{2d}(\widetilde{X\times X})
\end{equation}
We recall from (\ref{isom_taut}) that $c_1(\varphi^*\mathcal E_2)=q_*\tau^*pr_1^*H_{X} -\delta_X$. Intersecting (\ref{eq_rel_div_fin}) with $\pi_{|\widetilde{\Sigma^{[2]}}}^*(\Sigma^{[2]}\cdot c_1(\varphi^*\mathcal E_2)^{d-1})$ and projecting to $X$, we get in $\mathrm{CH}_d(X)$: \begin{equation}\label{eq_pres_proj} \begin{aligned} H_X\cdot \left[\widetilde{X^{[2]}}\right]_*(\Sigma^{[2]}\cdot c_1(\varphi^*\mathcal E_2)^{d-1})= \left[\widetilde{X^{[2]}}\right]_*\left[(2q_*\tau^*pr_1^*H_{X} - 3\delta_X)\cdot (\Sigma^{[2]}\cdot c_1(\varphi^*\mathcal E_2)^{d-1})\right] \\
-f'_{|\widetilde{X^{[2]}},*}(E_{P_2}\cdot\pi_{|\widetilde{X^{[2]}}}^*(\Sigma^{[2]}\cdot c_1(\varphi^*\mathcal E_2)^{d-1}))
\end{aligned}
\end{equation} 
To simplify this expression, we use the following lemma:
\begin{lemme}\label{lem_form_int} We have the following formulas (by induction): \\
\indent (i) for $k\geq 1$, $(q_*\tau^*pr_1^*H_{X})^k=\sum_{j=0}^{k-1} \binom {k-1} {j}q_*\tau^*(pr_1^*H^{k-j}_{X}\cdot pr_2^*H^j_{X})$;\\
\indent (ii) for $k, k'\geq 0$ and $m\geq 1$, $$q_*\tau^*(pr_1^*H^{k}_{X}\cdot pr_2^*H^{k'}_{X})\cdot \delta_X^m = q_*j_{E_X,*}[\tau_{|E_X}^*i_{\Delta_{X}}^*(pr_1^*H^{k}_{X}\cdot pr_2^*H^{k'}_{X})\cdot (E_{X|E_X})^{m-1}]$$ where $j_{E_X}:E_X\hookrightarrow \widetilde{X\times X}$ is the inclusion of the exceptional divisor, $i_{\Delta_{X}}:\Delta_{X}\hookrightarrow X\times X$ is the inclusion of the diagonal (so that $i_{\Delta_{X}}^*(pr_1^*H^{k}_{X}\cdot pr_2^*H^{k'}_{X})$ is the hyperplane section $H_{X}^{k+k'}$ on $X\simeq \Delta_{X}$) and $c_1(\mathcal O_{E_X}(-1))\simeq E_{X|E_X}$ is the tautological line bundle of the projective bundle $\tau_{|E_X}:E_X\rightarrow \Delta_{X}$.\\
\indent (iii) it follows that for $m\geq 2$, $$\begin{aligned}&c_1(\varphi^*\mathcal E_2)^m=\sum_{l=0}^{m-1}\binom{m-1}{l}q_*\tau^*(pr_1^*H_{X}^{m-l}\cdot pr_2^*H_{X}^l))  + (-1)^m\delta_X^m\\ &+ \sum_{k=1}^{m-1}\sum_{l=0}^{m-1-k}(-1)^k\binom{m}{k}\binom{m-1-k}{l}q_*j_{E_X,*}(\tau_{|E_X}^*i_{\Delta_X}^*(pr_1^*H_{X}^{m-k-l}\cdot pr_2^*H_{X}^l)\cdot E_{X|E_X}^{k-1})\end{aligned}$$
\end{lemme}

In order to establish (\ref{eq_inv_form_thm}), let us now compute the different terms of (\ref{eq_pres_proj}) modulo cycles coming from $\mathbb P^{n+1}$, using the correspondence $\sigma(\widetilde{X\times X})$. We recall that by construction, $[\sigma(\widetilde{X\times X})]_*(\cdot)=f'_{|\sigma(\widetilde{X\times X}),*}q^*(\cdot)$ and we have: 
$$\begin{tabular}{llll}
$q^*(\Sigma^{[2]}\cdot c_1(\varphi^*\mathcal E_2)^{d-1})$ &$=$ &$\widetilde{\Sigma\times\Sigma}\cdot [\sum_{l=0}^{d-2}\binom{d-2}{l}\tau^*(pr_1^*H_{X}^{d-1-l}\cdot pr_2^*H_{X}^l+pr_1^*H_{X}^{l}\cdot pr_2^*H_{X}^{d-1-l}) $\\
$ $ &$ $ &$+ (-1)^{d-1}E_X^{d-1} + \sum_{k=1}^{d-2}\sum_{l=0}^{d-2-k}(-1)^k\binom{d-1}{k}\binom{d-2-k}{l}j_{E_X,*}(\tau_{|E_X}^*H_X^{d-1-k}\cdot E_{X|E_X}^{k-1})]$
\end{tabular}$$
\textit{}\\
\indent The different terms are computed using the equalities:\\
\indent (i) for $m,m'\geq 0$, $\tau^*(\Sigma\times \Sigma)\cdot \tau^*(pr_1^*H_X^{m}\cdot pr_2^*H_X^{m'})=\tau^*((\Sigma\cap H_X^{m})\times (\Sigma\cap H_X^{m'}))$ and its image in $X$ under $f'_{|\sigma(\widetilde{X\times X}),*}$ is supported on $(\Sigma\cap H_X^{m})$.\\

\indent (ii) for $m\geq 0$, $\tau^*(\Sigma\times \Sigma)\cdot j_{E_X,*}(E_{X|E_X}^l\cdot\tau_{|E_X}^* H_X^{m})=j_{E_X,*}(E_{X|E_X}^l\cdot \tau_{|E_X}^*(\Sigma^2\cdot H_X^{m}))$ and its image in $X$ under $f'_{|\sigma(\widetilde{X\times X}),*}$ is supported on $\Sigma^2\cap H_X^{m}$.\\

\indent (iii) for $m>0$, $\tau^*(\Sigma\times \Sigma)\cdot E_X^m = j_{E_X,*}(E_{X|E_X}^{m-1}\cdot \tau_{|E_X}^*\Sigma^2)$ and its image in $X$ under $f'_{|\sigma(\widetilde{X\times X}),*}$ is supported on $\Sigma^2$.\\
\indent (iv) $j_{E_X,*}\{c(\tau^*_{|E_X}T_X)(1+E_{X|E_X})^{-1}\cdot \tau^*_{|E_X}i_{\Sigma,*}(c(T_\Sigma)^{-1})\}_{2d}\cdot \tau^*(pr_1^*H_X^{m}\cdot pr_2^*H_X^{m'})$\\
\indent \indent $ =j_{E_X,*}(\sum_{i=0}^{min(n-d,d)}(-1)^{n-i-1-d}E_{X|E_X}^{n-i-1-d}\cdot \tau_{|E_X}^*c_i(N_{\Sigma/X})))\cdot \tau^*(pr_1^*H_X^{m}\cdot pr_2^*H_X^{m'})$\\
\indent \indent $= 2j_{E_X,*}\sum_{i=0}^{min(n-d,d)}(-1)^{n-i-1-d}E_{X|E_X}^{n-i-1-d}\cdot \tau_{|E_X}^*(c_i(N_{\Sigma/X})\cdot H_X^{m+m'})$ and its image in $X$ under $f'_{|\sigma(\widetilde{X\times X}),*}$ is supported on $\cup_i (c_i(N_{\Sigma/X})\cap H_X^{m+m'})$.\\

\indent (v) $j_{E_X,*}\{c(\tau^*_{|E_X}T_X)(1+E_{X|E_X})^{-1}\cdot \tau^*_{|E_X}i_{\Sigma,*}(c(T_\Sigma)^{-1})\}_{2d}\cdot E_X^{m}$\\
\indent \indent $= j_{E_X,*}(\sum_{i=0}^{min(n-d,d)}(-1)^{n-i-1-d}E_{X|E_X}^{m+n-i-1-d}\cdot \tau_{|E_X}^*c_i(N_{\Sigma/X}))$ and its image in $X$ under $f'_{|\sigma(\widetilde{X\times X}),*}$ is supported on $\cup_i c_i(N_{\Sigma/X})$.\\

\indent (vi)$j_{E_X,*}\{c(\tau^*_{|E_X}T_X)(1+E_{X|E_X})^{-1}\cdot \tau^*_{|E_X}i_{\Sigma,*}(c(T_\Sigma)^{-1})\}_{2d}\cdot j_{E_X,*}(E_{X|E_X}^{k-1}\cdot \tau_{|E_X}^*H_X^{m})$\\
\indent \indent $=j_{E_X,*}\sum_{i=0}^{min(n-d,d)}(-1)^{n-i-1-d}E_{X|E_X}^{m+n-i-1-d}\cdot \tau_{|E_X}^*(c_i(N_{\Sigma/X})\cdot H_X^{m})$ and its image in $X$ under $f'_{|\sigma(\widetilde{X\times X}),*}$ is supported on $\cup_i (c_i(N_{\Sigma/X})\cap H_X^{m})$.\\ \\

With these formulas, we can see that:\\
\indent (1) We have $[\sigma(\widetilde{X\times X})]_*(\Sigma^{[2]}\cdot c_1(\varphi^*\mathcal E_2)^{d-1})= 0$ as its support in $X$ is the union of subvarieties of dimension $\leq d$ whereas $q^*(\Sigma^{[2]}\cdot c_1(\varphi^*\mathcal E_2)^{d-1})$ has dimension $d+1$. So $\left[\widetilde{X^{[2]}}\right]_*(\Sigma^{[2]}\cdot c_1(\varphi^*\mathcal E_2)^{d-1})\in \mathbb Z\cdot H_X^{n-d-1}$.\\ \\

\indent (2) We have $$[\sigma(\widetilde{X\times X})]_*(\Sigma^{[2]}\cdot c_1(\varphi^*\mathcal E_2)^{d})= f_{|\sigma(\widetilde{X\times X}),*}[\tau^*(\Sigma\times (\Sigma\cdot H_X^{d})) - j_{E_X,*}((-1)^{n-1}E_{X|E_X}^{n-1}\cdot \tau_{|E_X}^*c_0(N_{\Sigma/X})]$$ since all the other terms are supported on $\bigcup_{k,j,i,m}(\Sigma\cap H_X^{k})\cup (\Sigma^2\cap H_X^{j}) \cup (c_i(N_{\Sigma/X})\cap H_X^{m})$ with $k>0$, $j\geq 0$ and $m>0$ if $i=0$ and $m\geq 0$ else, which is a union of subschemes of dimension $<d$. So $[\sigma(\widetilde{X\times X})]_*(\Sigma^{[2]}\cdot c_1(\varphi^*\mathcal E_2)^{d})= deg(\Sigma)\Sigma - \Sigma$ in $\mathrm{CH}_d(X)$. Hence $\left[\widetilde{X^{[2]}}\right]_*(\Sigma^{[2]}\cdot c_1(\varphi^*\mathcal E_2)^{d})=-(deg(\Sigma)\Sigma -\Sigma)\ \mathrm{mod}\ \mathbb Z\cdot H_X^{n-d}$.\\ \\

\indent (3) Likewise $[\sigma(\widetilde{X\times X})]_*(\Sigma^{[2]}\cdot c_1(\varphi^*\mathcal E_2)^{d-1}\cdot \delta_X) = f_{|\sigma(\widetilde{X\times X}),*}((-1)^{n}E_{X|E_X}^{n-1}\tau_{|E_X}^*c_0(N_{\Sigma/X}))$ so that $\left[\widetilde{X^{[2]}}\right]_*(\Sigma^{[2]}\cdot c_1(\varphi^*\mathcal E_2)^{d-1})\cdot \delta_X)= \Sigma\ \mathrm{mod}\ \mathbb Z\cdot H_X^{n-d}$.\\ \\

\indent (4) For the last term, we have $$f'_{|\widetilde{X^{[2]}},*}(E_{P_2}\cdot\pi_{|\widetilde{X^{[2]}}}^*(\Sigma^{[2]}\cdot c_1(\varphi^*\mathcal E_2)^{d-1})) = P_*[p_{P_2,*}i_{P_2}^*(\Sigma^{[2]}\cdot c_1(\varphi^*\mathcal E_2)^{d-1})].$$
\end{proof}
\textit{}\\

\subsection{A digression on  the  Hilbert square of subvarieties} Assume $k=\mathbb C$. On one hand, as any smooth cubic hypersurface $X$ admits a unirational parametrization of degree $2$, any functorial birational invariant of $X$ is $2$-torsion and as the coefficient appearing with $\Sigma$ in the inversion formula of Theorem \ref{thm_form_single}, is odd, the formula will be useful to study birational invariants obtained as functorial subquotient of Chow groups. On the other hand, in the inversion formula, the operation $\Sigma\mapsto \Sigma^{[2]}$ plays a key role. So let us look at some properties of this operation.

\begin{proposition}\label{prop_hilb2} Let $Y$ be a smooth projective $k$-variety. Let $V,V'$ be smooth subvarieties of $Y$ of dimension $d<dim(Y)$ and $N>0$ an integer such that $N(i_{V,*}(c(V)^{-1})-i_{V',*}(c(V')^{-1})) =0$ in $\mathrm{CH}_*(Y)$ (resp. $\mathrm{CH}_*(Y)/alg$), where $i_V$ (resp. $i_{V'}$) is the inclusion of $V$ (resp. of $V'$) in $Y$.\\
\indent (i) Then $2N(V^{[2]}-V'^{[2]})=0$ in $\mathrm{CH}_{2d}(Y^{[2]})$ (resp. $\mathrm{CH}_{2d}(Y^{[2]})/alg$).\\
\indent (ii) Moreover if the groups $\mathrm{CH}_{i}(Y)$ are torsion-free for $i\leq 2d$, then $V^{[2]}=V'^{[2]}$ in $\mathrm{CH}_{2d}(Y^{[2]})$.
\end{proposition}
\begin{proof} (i) Let us denote $\tau:\widetilde{Y\times Y}\rightarrow Y\times Y$ the blow-up of $Y\times Y$ along the diagonal $\Delta_X$ and $q:\widetilde{Y\times Y}\rightarrow Y^{[2]}$ the quotient by the involution. For a smooth subvariety $V\subset Y$, we have $q_*(\widetilde{V\times V}^{\Delta_V})=2V^{[2]}$ in $\mathrm{CH}_{2d}(Y^{[2]})$, where $\widetilde{V\times V}^{\Delta_V}$ is the blow-up of $V\times V$ along its diagonal $\Delta_V$, i.e. the proper transform of $V\times V$ under $\tau$. We recall (\cite[Theorem 6.7 and Corollary 4.2.1]{Fulton}) that we have
\begin{equation}\label{eq_strict_tr}
\widetilde{V\times V}^{\Delta_V} = \tau^*(V\times V) - j_{E_Y,*}\{c(\tau^*_{|E_Y}T_Y)c(E_{Y|E_Y})^{-1}\cdot \tau^*_{|E_Y}i_{V,*}(c(T_V)^{-1})\}_{2d}\ \mathrm{in\ CH}_{2d}(\widetilde{Y\times Y})
\end{equation}
where $j_{E_Y}:E_Y\hookrightarrow \widetilde{Y\times Y}$ is the exceptional divisor of $\tau$ and for an element $z\in \mathrm{CH}_*(\widetilde{Y\times Y})$, $\{z\}_k$ is the part of dimension $k$ of $z$.\\
\indent Now, if $N(V-V')=0$ in $\mathrm{CH}_d(Y)$ (resp. $\mathrm{CH}_d(Y)/alg$), $V$ and $V'$ being smooth subvarieties of $Y$, then $$N(V\times V)=Npr_1^*V\cdot pr_2^*V=pr_1^*(NV)\cdot pr_2^*V= pr_1^*(NV')\cdot pr_2^*V = pr_1^*V'\cdot pr_2^*(NV)=N(V'\times V')$$ in $\mathrm{CH}_{2d}(Y\times Y)$ (resp. $\mathrm{CH}_{2d}(Y\times Y)/alg$). So we see that the hypothesis yields
$$\begin{tabular}{llll}
$2N(V^{[2]}-V'^{[2]})$ &$=$ &$Nq_*(\widetilde{V\times V}^{\Delta_V} - \widetilde{V'\times V'}^{\Delta_{V'}})$\\
$ $ &$=$ &$\tau^*N[(V\times V) - (V'\times V')] $\\
$ $ &$ $ &$- j_{E_Y,*}\{c(\tau^*_{|E_Y}T_Y)c(E_{Y|E_Y})^{-1}\cdot\tau_{|E_Y}^*N((i_{V,*}(c(T_V)^{-1})-i_{V',*}(c(T_{V'})^{-1}))\}_{2d}$\\
$ $ &$=$ &$0\ \  \mathrm{in\ CH}_{2d}(Y^{[2]})\ \mathrm{(resp.\ CH}_{2d}(Y^{[2]})/alg\mathrm{).}$
\end{tabular}$$\\

\indent (ii) As $\mathrm{CH}_{*\leq d}(Y)$ is assumed to be torsion-free, $V=V'$ in $\mathrm{CH}_d(Y)$. Then, by \cite[Proposition 1.4]{Moon_Poli}, $V^{(2)}=V'^{(2)}$ in $\mathrm{CH}_{2d}(Y^{(2)})$, where, for a variety $Z$, $Z^{(2)}$ is the symmetric product of $Z$. We have the localisation exact sequence $$\mathrm{CH}_{2d}(E_Y)\rightarrow \mathrm{CH}_{2d}(Y^{[2]})\rightarrow \mathrm{CH}_{2d}(Y^{[2]}\backslash E_Y)\rightarrow 0$$ and since $Y^{[2]}\backslash E_Y\simeq Y^{(2)}\backslash \Delta_Y$, $V^{[2]}-V'^{[2]}$ can be written $q_*j_{E_Y,*}\gamma$ for a $2d$-cycle $\gamma\in \mathrm{CH}_{2d}(E_Y)$. According to item (i), $2(V^{[2]}-V'^{[2]})=0$ so that $q_*j_{E,*}(2\gamma)=0$. As, $q$ is flat, $q^*q_*j_{E,*}(2\gamma)=[q^{-1}q_*j_{E,*}(2\gamma)]=j_{E,*}(2\gamma)$ and by the decomposition of the Chow groups of the blow-up $\widetilde{Y\times Y}$, $2\gamma=0$ in $\mathrm{CH}_{2d}(E_Y)$. So, by the decomposition of the Chow groups of projective bundle and torsion-freeness of $\mathrm{CH}_{*\leq 2d}(Y)$, $\gamma=0$ i.e. $V^{[2]}-V'^{[2]}=0$ in $\mathrm{CH}_{2d}(Y^{[2]})$.
\end{proof}
\indent Unfortunately, in general, for a smooth subvariety $V$ of a smooth projective variety $Y$, one cannot expect the class of $V^{[2]}$ in $\mathrm{CH}_*(Y^{[2]})$ to be determined by $(i_{V,*}(c(T_V)^{-1}))$ as the following example, which was communicated to the author by Voisin, shows.\\
\indent Let $S$ be an abelian surface and $x,\,y\in S$ be two distinct $2$-torsion points.
 For any sufficiently ample linear system $\mathcal L$ on
 $S$, there exists a curve  $C_x\in |\mathcal L|$ not containing $y$, resp.  $C_y\in |\mathcal L|$ not containing $x$, which is smooth away from $x$, resp. $y$, and
 has an ordinary double point at $x$, resp. $y$. Let $\tau:\widetilde{S}\rightarrow S$ be the blow-up of $S$ at $x$ and $y$ and $E_x,E_y$ the corresponding exceptional divisors. The normalization $\widetilde{C_x}$ (resp. $\widetilde{C_y}$) of $C_x$ (resp. $C_y$) is the strict transform of $C_x$ (resp. $C_y$) under $\tau$ and its class in $\mathrm{Pic}(\widetilde{S})$ is $\tau^*c_1(\mathcal L) -2E_x$ (resp. $\tau^*c_1(\mathcal L) -2E_y$).\\
\indent Let $\mathcal N\in \mathrm{Pic}(S)$ be sufficiently ample on $S$ so that the line bundle $\tau_{|\widetilde{C_x}}^*\mathcal N_{|{C_x}}$ is very ample (once its degree on $\widetilde{C}_x$ is large enough) on $\widetilde{C_x}$ and $\tau_{|\widetilde{C_y}}^*\mathcal N_{|{C_y}}$ is very ample on $ \widetilde{C_y}$. We can pick a meromorphic function $f_x:\widetilde{C_x}\rightarrow \mathbb P^1$ in $|\tau_{|\widetilde{C_x}}^*\mathcal N_{|{C_x}}|$ such that, denoting $x_1$ and $x_2$ the points lying over the node $x$, $f_x(x_1)\neq f_x(x_2)$. Likewise, we can pick a meromorphic function $f_y:\widetilde{C_y}\rightarrow \mathbb P^1$ in $|\tau_{|\widetilde{C_y}}^*\mathcal N_{|{C_y}}|$ such that $f_y(y_1)\neq f_y(y_2)$, where $y_1,y_2$ are the points lying over the node $y$.\\
\indent Let $X= S\times \mathbb P^1$  be the trivial projective bundle over $S$.  By construction the morphisms $(\tau_{|\widetilde{C_x}},f_x):\widetilde{C_x}\rightarrow X$ and $(\tau_{|\widetilde{C_y}},f_y):\widetilde{C_y}\rightarrow X$ are embeddings so $D_x=(\tau_{|\widetilde{C_x}},f_x)(\widetilde{C_x})$ and $D_y=(\tau_{|\widetilde{C_y}},f_y)(\widetilde{C_y})$ are smooth curves on $X$.
\begin{proposition}\label{prop_ex_hilb2} In this situation, we have $i_{D_x,*}(c(T_{D_x})^{-1})=i_{D_y,*}(c(T_{D_y})^{-1})$ in $\mathrm{CH}_*(X)$ but $D_x^{[2]}\neq D_y^{[2]}$ in $\mathrm{CH}_2(X^{[2]})$.
\end{proposition}
\begin{proof} We have the decomposition $\mathrm{CH}_1(X)\simeq pr_1^*\mathrm{CH}_0(S)\oplus pr_1^*\mathrm{CH}_1(S)$. The projection on $\mathrm{CH}_1(S)$ is given by $pr_{1,*}$; we have $pr_{1,*}D_x=\tau_{|\widetilde{C_x},*}\widetilde{C_x}=C_x$ and $pr_{1,*}D_y=\tau_{|\widetilde{C_y},*}\widetilde{C_y}=C_y$ and $C_x,C_y\in |\mathcal L|$. As the Chern classes of the trivial bundle are trivial the projection on $\mathrm{CH}_0(S)$ is given by the composition of the intersection with $pr_2^*c_1(\mathcal O_{\mathbb P^1}(1))$ followed by $pr_{1,*}$. We have $f_x^*\mathcal O_{\mathbb P^1}(1)\simeq \tau_{|\widetilde{C_x}}^*\mathcal N_{|{C_x}}$ and using projection formula and $C_x\in |\mathcal L|$, we get $pr_{1,*}(D_x\cdot pr_2^*c_1(\mathcal O_{\mathbb P^1}(1)))= c_1(\mathcal L)\cdot c_1(\mathcal N)$ in $\mathrm{CH}_0(S)$. Likewise, we have $pr_{1,*}(D_y\cdot pr_2^*c_1(\mathcal O_{\mathbb P^1}(1)))= c_1(\mathcal L)\cdot c_1(\mathcal N)$. So $D_x=D_y$ in $\mathrm{CH}_1(X)$.\\
\indent By adjunction, we have $K_{\widetilde{C_x}}=(K_{\widetilde{S}}+\widetilde{C_x})_{|\widetilde{C_x}} = (\tau^*(c_1(\mathcal L)+K_S)+E_y-E_x)_{|\widetilde{C_x}}$ so that in $\mathrm{CH}_0(X)\simeq \mathrm{CH}_0(S)$,
$$\begin{tabular}{lll}
$i_{D_x,*}K_{D_x}$ &$=$ &$(\tau\circ i_{\widetilde{C_x}})_*i_{\widetilde{C_x}}^*(\tau^*c_1(\mathcal L)+E_y-E_x)$\\
$ $ &$=$ &$\tau_*i_{\widetilde{C_x},*}i_{\widetilde{C_x}}^*(\tau^*c_1(\mathcal L)+E_y-E_x)$\\
$ $ &$=$ &$\tau_*((\tau^*c_1(\mathcal L)+E_y-E_x)\cdot \widetilde{C_x})$\\
$ $ &$=$ &$\tau_*(\tau^*(c_1(\mathcal L)^2 +2E_x^2)$\\
$ $ &$=$ &$c_1(\mathcal L)^2-2x$
\end{tabular}$$
Likewise $i_{D_y,*}K_{D_y}= c_1(\mathcal L)^2 - 2y$. As $2x=2y$ in $\mathrm{CH}_0(S)$, $i_{D_x,*}K_{D_x}=i_{D_y,*}K_{D_y}$ in $\mathrm{CH}_0(X)$. So $i_{D_x,*}(c(T_{D_x})^{-1})=i_{D_y,*}(c(T_{D_y})^{-1})$.\\
\indent The variety of lines of $X$, with respect to a very ample line bundle of the form $pr_1^*\mathcal L'\otimes pr_2^*\mathcal O_{\mathbb P^1}(1)$, is isomorphic to $S$ since any morphism from a projective space to a abelian variety is constant. Let us denote $P_2=\mathbb P({\rm Sym}^2\mathcal E)\simeq S\times \mathbb P^2$; it parametrizes the length $2$ subschemes of $X$ contained in a line of $X$. So $D_x^{[2]}\cap P_2$ parametrizes length $2$ subschemes of $D_x$ such that the associated line is contained in $X$. But by construction, since $pr_{1,|D_x}:D_x\rightarrow C_x$ is an isomorphism above $C_x\backslash \{x\}$, the only length $2$ subscheme whose associated line is contained in $X$ is $\{x_1+x_2\}$ whose associated line is $\mathbb P(\mathcal E_x)$. So, denoting $i_{P_2}:P_2\hookrightarrow X^{[2]}$ the natural inclusion and $\pi_1:P_2\rightarrow S$ the first projection, we have $\pi_{1,*}(i_{P_2}^*D_x^{[2]})=x$ in $\mathrm{CH}_0(S)$.\\
\indent Likewise $\pi_{1,*}(i_{P_2}^*D_y^{[2]})=y$ in $\mathrm{CH}_0(S)$. So $\pi_{1,*}i_{P_2}^*(D_x^{[2]}-D_y^{[2]})=x-y\neq 0$ in $\mathrm{CH}_0(S)$, in particular $D_x^{[2]}-D_y^{[2]}$ is a nonzero $2$-torsion element in $\mathrm{CH}_2(X^{[2]})$.
\end{proof}

\subsection{Application of the inversion formula\label{secderniere}} Using the results of the previous sections, we get the following:
\begin{theoreme}\label{thm_torsion} Let $X\subset \mathbb P_{\mathbb C}^{n+1}$, with $n\geq 5$, be a smooth cubic hypersurface. For any $\Gamma\in \mathrm{CH}_2(X)$ of $t$-torsion (hence homologically trivial), there is a homologically trivial $2t$-torsion $1$-cycle $\gamma\in \mathrm{CH}_1(F(X))$ and an odd integer $m$ such that $m\Gamma = P_*\gamma$ in $\mathrm{CH}_2(X)$.
\end{theoreme}
\begin{proof} Let $\Gamma\in \mathrm{CH}_2(X)$ be a cycle annihilated by $t\in \mathbb{Z}_{>0}$. Using Proposition \ref{prop_ruled_hyperpl}, we can find a $1$-cycle $\alpha$ in $F(X)$ such that $P_*(\alpha)=\Gamma$. As $\Gamma$ is a torsion cycle and $\mathrm{CH}_0(X)=\mathbb Z$, $\Gamma\cdot H^2=0$ and since $c_1(\mathcal O_{\mathbb P(\mathcal E_{2|F(X)})}(1))=q^*H$, we get $deg(\alpha\cdot c_1(\mathcal O_{F(X)}(1)))=deg(q_*[p^*(\alpha\cdot c_1(\mathcal O_{F(X)}(1))\cdot q^*H])=0$, where $\mathcal O_{F(X)}(1)=det(\mathcal E_{2|F(X)})$ is the Pl\"ucker line bundle, which implies that $\alpha\cdot c_1(\mathcal O_{F(X)}(1))=0$ in $\mathrm{CH}_0(F(X))$ since $F(X)$ is rationally connected. As $\mathrm{Pic}(F(X))\simeq \mathbb Z$ (\cite[Corollaire 3.5]{Deb-Man}), $\alpha$ is numerically trivial. We have the following lemma:
\begin{lemme}\label{prop_mm_genre} Let $Y$ be a smooth projective variety of dimension $d\geq 3$ and $D$ a numerically trivial $1$-cycle of $Y$. Then there are smooth curves $D_1,D_2\subset Y$ of the same genus such that $D=D_1-D_2$ in $\mathrm{CH}_1(Y)$.
\end{lemme}
Postponing the proof of the lemma, we conclude as follows: let $E_1,E_2\subset F(X)$ be two smooth curves of genus $g$ such that $\alpha=E_1-E_2$ in $\mathrm{CH}_1(F(X))$; they have also the same degree ($\alpha$ is numerically trivial) that we shall denote $d$. Let us denote $S_{E_1}=q(p^{-1}(E_1))$ and $S_{E_2}=q(p^{-1}(E_2))$ the associated ruled surfaces in $X$, we have $\Gamma=P_*\alpha=S_{E_1}-S_{E_2}$in $\mathrm{CH}_2(X)$. By transversality arguments, we can arrange that $q$ induces an embedding of $p^{-1}(E_1)$ (resp. $p^{-1}(E_2)$) in $X$ so that $S_{E_1}$ (resp. $S_{E_2}$) is smooth and isomorphic to $p^{-1}(E_1)$ (resp. $p^{-1}(E_2)$). An easy computation then gives:
$$\begin{tabular}{lll}
$i_{S_{E_1},*}(c(T_{S_{E_1}})^{-1})$ &$=$ &$S_{E_1} + P_*(i_{E_1,*}K_{E_1} + dc_1(\mathcal O_{F(X)}(1))\cdot E_1) - 2S_{E_1}\cdot H_X - 2H_X\cdot P_*(i_{E_1,*}K_{E_1})$\\
$ $ &$=$ &$P_*(E_1) + (2g-2+d)P_*[l_0] - 2 P_*(E_1)\cdot H_X - 2(2g-2)P_*([l_0])\cdot H_X$
\end{tabular}$$
since $\mathrm{CH}_0(F(X))\simeq \mathbb Z\cdot [l_0]$ for a (any) point $[l_0]\in F(X)$. Likewise $$i_{S_{E_2},*}(c(T_{S_{E_2}})^{-1})=P_*(E_2) + (2g-2+d)P_*[l_0] - 2 P_*(E_2)\cdot H_X - 2(2g-2)P_*([l_0])\cdot H_X$$ so that $i_{S_{E_1},*}(c(T_{S_{E_1}})^{-1})-i_{S_{E_2},*}(c(T_{S_{E_2}})^{-1})=(S_{E_1} -S_{E_2})\cdot (1-2H_X)$ is annihilated by $t$ in $\mathrm{CH}_*(X)$. Using Proposition \ref{prop_hilb2}, we get that $S_{E_1}^{[2]}-S_{E_2}^{[2]}$ is annihilated by $2t$ in $\mathrm{CH}_4(X^{[2]})$. According to \cite[Theorem 2.2]{Totaro_hilb_sch}, since $H^*(X,\mathbb Z)$ is torsion-free (by Lefschetz hyperplane and universal coefficient theorems), $H^*(X^{[2]},\mathbb Z)$ is torsion-free so that $[S_{E_1}^{[2]}-S_{E_2}^{[2]}]=0$ in $H^{4n-8}(X^{[2]},\mathbb Z)$.\\
\indent Now, Theorem \ref{thm_form_single}, says that there are integers $m_1, m_2$ such that $$(2d-3)S_{E_1} + P_*(p_{P_2,*}i_{P_2}^*S_{E_1}^{[2]}\cdot c_1(\mathcal O_{F(X)}(1)))= m_1 H_X^{n-2}\ \ \mathrm{in\ CH}_2(X)$$ and  $$(2d-3)S_{E_2} + P_*(p_{P_2,*}i_{P_2}^*S_{E_2}^{[2]}\cdot c_1(\mathcal O_{F(X)}(1)))= m_2 H_X^{n-2}\ \ \mathrm{in\ CH}_2(X)$$ in particular $(2d-3)\Gamma + P_*(p_{P_2,*}i_{P_2}^*(S_{E_1}^{[2]}-S_{E_2}^{[2]})\cdot c_1(\mathcal O_{F(X)}(1))) \in \mathbb Z\cdot H_X^{n-2}$. But intersecting with $H_X^2$, since $\Gamma$ and $p_{P_2,*}i_{P_2}^*(S_{E_1}^{[2]}-S_{E_2}^{[2]})$ are torsion cycles, we see that actually:$$(2d-3)\Gamma + P_*(p_{P_2,*}i_{P_2}^*(S_{E_1}^{[2]}-S_{E_2}^{[2]})\cdot c_1(\mathcal O_{F(X)}(1)))=0$$ in $\mathrm{CH}_2(X)$. Moreover $p_{P_2,*}i_{P_2}^*(S_{E_1}^{[2]}-S_{E_2}^{[2]})$ is homologically trivial since $S_{E_1}^{[2]}-S_{E_2}^{[2]}$ is.
\end{proof}

\begin{proof}[Proof of Lemma \ref{prop_mm_genre}.] Using Hironaka's smoothing of cycles (\cite{Hir_sm_cycles}) and moving lemma, we can write $D=\sum_im_iC_i$ where $(C_i)_{1\leq i\leq N}$ is a family of smooth pairwise disjoint connected curves. We can always assume that there is a $i_0$ such that $m_{i_0}=1$. Indeed, if none of the $m_i$ is equal to $1$, then we can pick $2$ smooth curves $C_{N+1}, C_{N+2}\subset Y$ which are rationally equivalent such that $(C_i)_{1\leq i\leq N+2}$ is still a family of pairwise disjoint smooth curves. Then $D=\sum_im_iC_i+C_{N+1}-C_{N+2}$ in $\mathrm{CH}_1(Y)$.\\
\indent Let $C\subset Y$ be a smooth curve intersecting $C_{i_0}$ transversally in a unique point and disjoint from the remaining $C_i$ and $Z=(\cup_{i=1}^NC_i)\cup C$. The subscheme $Z$ is purely $1$-dimensional and smooth away from the point $C\cap C_{i_0}$ which is an ordinary double point. In particular $Z$ is a local complete intersection subscheme, so that the sheaf $I_Z/I_Z^2$ on $Z$ is a vector bundle that we shall denote $N_{Z/Y}^\vee$. Let $\mathcal L\in \mathrm{Pic}(Y)$ be a very ample line bundle such that $H^1(Y,\mathcal L\otimes I_Z^2)=0$ and $N_{Z/Y}^\vee\otimes \mathcal L_{|Z}$ is globally generated. Then, from the exact sequence $$0\rightarrow I_Z^2\rightarrow I_Z\rightarrow N_{Z/Y}^\vee\rightarrow 0$$ we get a surjective morphism $H^0(Y,\mathcal L\otimes I_Z)\stackrel{\rho}{\rightarrow} H^0(Z,N_{Z/Y}^\vee\otimes \mathcal L_{|Z})$. According to \cite[Lemma 1]{Schnell_hypersurf}, for any nonzero section $s\in H^0(Y,\mathcal L\otimes I_Z)$, the zero scheme $V(s)\subset Y$ is singular at a point $x\in Z$ if and only if the section $\rho(s)$ of $N_{Z/Y}^\vee\otimes \mathcal L_{|Z}$ vanishes at $x$. As, $N_{Z/Y}^\vee\otimes \mathcal L_{|Z}$ is globally generated of rank $\geq 2$, the zero locus of a generic section of $N_{Z/Y}^\vee\otimes \mathcal L_{|Z}$ has codimension $rank(N_{Z/Y}^\vee\otimes \mathcal L_{|Z})\geq 2$ i.e. is empty. So we can find a smooth hypersurface in $|\mathcal L|$ containing $Z$. Repeating the process, we can get a smooth surface $S\subset Y$, which is complete intersection of hypersurfaces given by sections of powers of $\mathcal L$, containing $Z$. Next it is a standard fact (e.g consequence of   Riemann-Roch formula) that for any divisor $W$ on a smooth projective surface $S$, $deg(W\cdot(W+K_S))$ is even.
Applying this fact  to the divisor $D=\sum_im_iC_i$ of $S$, $\int_S D\cdot(D+K_S)$ is even and since $D$ is numerically trivial on $Y$ and $K_S$ is the restriction of a divisor of $Y$ by adjunction formula ($S$ is complete intersection in $Y$), $deg(D^2)\in 2\mathbb Z$. Let us write $deg(D^2)=2\ell$. Let $H\in \mathrm{Pic}(S)$ be a very ample divisor coming from $Y$ such that the line bundles $\mathcal O_S(H-\ell C)$ and $\mathcal O_S(H-\ell C+D)$ are ample. We can choose smooth connected curves $E_1\in |H-\ell C+D|$ and $E_2\in |H-\ell C|$; we then have $D=E_1-E_2$ in $\mathrm{Pic}(S)$ (thus, in $\mathrm{CH}_1(Y)$ also). By adjunction formula, we have:
$$\begin{tabular}{lll}
$2g(E_1)-2$ &$=$ &$\int_S (H-\ell C+D)\cdot(H-\ell C+D+K_S)$\\
$ $ &$=$ &$\int_S(H-\ell C)\cdot(H-\ell C+K_S) + \int_SD\cdot(H-\ell C+D+K_S) + \int_SD\cdot(H-\ell C)$\\
$ $ &$=$ &$\int_S(H-\ell C)\cdot(H-\ell C+K_S) + \int_SD^2 -2\ell D\cdot C$\\
$ $ &$ $ &$\ \ \  \mathrm{since}\ D\ \mathrm{is\ numerically\ trivial\ on}\ Y\ \mathrm{and}\ H\ \mathrm{and}\ K_S\ \mathrm{come\ from\ divisors\ of}\ Y$\\
$ $ &$=$ &$\int_S(H-\ell C)\cdot(H-\ell C+K_S) \ \mathrm{since\ by\ construction}\ \int_SC\cdot D= \int_SC\cdot C_{i_0}=1$
\end{tabular}$$
 and
$$2g(E_2)-2= \int_S (H-\ell C)\cdot(H-\ell C+K_S).$$ i.e. $g(E_1)=g(E_2)$.
\end{proof}
Before stating our main corollary, let us prove the following lemma:
\begin{lemme} The group $\mathrm{CH}_1(Y)_{tors,AJ}$ is a stable birational invariant for smooth projective  varieties
$Y$.
\end{lemme}
\begin{proof} As usual, it suffices to prove invariance under taking products with
$\mathbb{P}^r$and under blow-ups. We have
$${\rm CH}_1(Y\times \mathbb{P}^r)={\rm CH}_1(Y)\oplus {\rm CH}_0(Y)$$
and this decomposition is compatible with the Deligne cycle class map. As the torsion
of ${\rm CH}_0(Y)$ injects into ${\rm Alb}\,Y$ by Roitman \cite{Roitman}, it follows
that ${\rm CH}_0(Y)_{tors, AJ}=0$ which proves the first invariance.
Similarly, let $\widetilde{Y}_Z\rightarrow Y$ be the blow-up of
$Y$ along $Z$, with $Z$ smooth of codimension $\geq2$. Then we have
$${\rm CH}_1(\widetilde{Y}_Z)={\rm CH}_1(Y)\oplus {\rm CH}_0(Z)$$
and we conclude by the same argument invoking \cite{Roitman} that
${\rm CH}_1(\widetilde{Y}_Z)_{tors, AJ}={\rm CH}_1({Y})_{tors, AJ}$.
\end{proof}
\begin{remarque} {\rm In fact the same arguments show that the group ${\rm CH}_1({Y})_{tors, AJ}$
is trivial when $Y$ admits a Chow-theoretic decomposition of the diagonal.}
\end{remarque}

We have the following corollary for cubic $5$-folds:
\begin{corollaire}\label{cor_tors_abel_j} Let $X\subset \mathbb P^6_{\mathbb C}$ be smooth cubic hypersurface. Then $P_*:\mathrm{CH}_1(F(X))_{tors,AJ}\rightarrow \mathrm{CH}^3(X)_{tors,AJ}$ is surjective. So the birational invariant $\mathrm{CH}^3(X)_{tors,AJ}$ of $X$ is controlled by the group $\mathrm{CH}_1(F(X))_{tors,AJ}$.
\end{corollaire}
\begin{proof} Let $\Gamma\in \mathrm{CH}^3(X)_{tors, AJ}\simeq \mathrm{CH}_2(X)_{tors,AJ}$; by Theorem \ref{thm_torsion}, there are a homologically trivial torsion cycle $\gamma\in \mathrm{CH}_1(F(X))$ and an integer $d$ such that $(2d-3)\Gamma=P_*\gamma$. Because of the degree $2$ unirational parametrization of $X$, $\mathrm{CH}^3(X)_{tors, AJ}$ is a $2$-torsion group; in particular $(2d-3)\Gamma=\Gamma$ in $\mathrm{CH}^3(X)$ and it is equal to $P_*\gamma$. By functoriality of Abel-Jacobi maps ($P_*$ induces morphisms of Hodge structures), denoting, for a complex variety $Y$, $\Phi_Y^{2c-1}$ the Abel-Jacobi map for homologically trivial cycles of codimension $c$, $\Phi_X^5(\Gamma)=P_*\Phi_{F(X)}^9(\gamma)$. Now, by \cite{Shimada}, $P_*$ is an isomorphism of abelian varieties so $\gamma$ is annihilated by the Abel-Jacobi map.
\end{proof}

Under the assumption of the corollary, the variety $F(X)$ is Fano, hence rationally connected. Along the lines of the questions asked in \cite{vois_survey} for the group ${\rm Griff}_1(Y)$, and
the results proved in  \cite{Vois_rem_rat_curv} for
the group $H_2(Y,\mathbb{Z})/H_2(Y,\mathbb{Z})_{alg}$ (showing that it should be trivial
for rationally connected varieties), it is tempting to believe that the group $\mathrm{CH}_1(Y)_{tors,AJ}$
is always trivial for rationally connected varieties. Thus we can see
Corollary \ref{cor_tors_abel_j} as  an evidence that the group
$\mathrm{CH}^3(X)_{tors,AJ}$ should be trivial.\\
\indent For example, for cubic hypersurfaces, we have the following result which follows essentially from the work of Shen (\cite{Shen_main}):
\begin{proposition}\label{Prop_ex_4-fold} Let $X\subset \mathbb P^{n+1}_{\mathbb C}$, with $n\geq 3$, be a smooth cubic hypersurface. Then $\mathrm{CH}_1(X)_{tors,AJ}=0$
\end{proposition}
\begin{proof} For cubic threefolds, the proposition can be obtained as a consequence of the work of Bloch and Srinivas (\cite[Theorem 1]{Bl-Sr}) which asserts that $\mathrm{CH}_1(X)_{hom}\simeq \mathrm{CH}_1(X)_{alg}\simeq J^3(X)(\mathbb C)$. For cubic hypersurfaces of dimension $\geq 5$, the result follows from the work of Shen (\cite{Shen_main}) who proved that $\mathrm{CH}_1(X)\simeq \mathbb Z$. The only case left is the case of cubic $4$-folds but the following proof works for cubic hypersurfaces of any dimension $\geq 3$.\\
\indent Pick $\gamma\in \mathrm{CH}_1(X)_{tors,AJ}$. It is a numerically trivial $1$-cycle of $X$ so according to Lemma \ref{prop_mm_genre} we can write it as $\gamma=C_1-C_2$ where $C_i$ are smooth connected curves on $X$ of same genus $g$ and same degree $d$. Applying the inversion formula to $C_i$ yields: \begin{equation}\label{eq_rq} (2d-3)\gamma = P_*p_{P_2,*}i_{P_2}^*(C_2^{(2)}-C_1^{(2)})\ \mathrm{in\ CH}_1(X).
\end{equation} Since the $C_i$ have the same genus and $\gamma=C_1-C_2$ is torsion, by Proposition \ref{prop_hilb2}, $C_2^{(2)}-C_1^{(2)}\in \mathrm{CH}_2(X^{[2]})$ is torsion. As $H^*(X^{[2]},\mathbb Z)$ is torsion-free (\cite{Totaro_hilb_sch}), $C_2^{(2)}-C_1^{(2)}$ is homogically trivial. When $n=3$, $P_*$ yields an isomorphism $\mathrm{Alb}(F(X))\simeq J^3(X)$ (\cite{Cl-Gr}), so that applying Abel-Jacobi map to (\ref{eq_rq}), we see that $p_{P_2,*}i_{P_2}^*(C_2^{(2)}-C_1^{(2)})\in \mathrm{CH}_0(F(X))_{tors,AJ}$. When $n\geq 4$, as the Albanese variety of $F(X)$ is trivial (\cite[Proposition 3]{BD}, \cite{Deb-Man}), $p_{P_2,*}i_{P_2}^*(C_2^{(2)}-C_1^{(2)})\in \mathrm{CH}_0(F(X))_{tors,AJ}$ and we conclude by Roitman theorem (\cite{Roitman}).
\end{proof}

\section*{Acknowledgments}
I am grateful to my advisor Claire Voisin for having brought to me this interesting question, as well as her kind help and great and patient guidance during this work. I am also grateful to Mingmin Shen for pointing me a deficiency in the original proof of Theorem \ref{thm_torsion}. I would like to thank Jean-Louis Colliot-Th\'el\`ene for sharing with me the reference \cite{KS} where the vanishing of $H_{nr}^4(X,\mathbb Q/\mathbb Z)$ when $X$ is a cubic $5$-fold is proved. Finally, I am grateful to the gracious Lord for His care. 

\noindent \begin{tabular}[t]{l}
\textit{rene.mboro@polytechnique.edu}\\
CMLS, Ecole Polytechnique, CNRS, Universit\'e Paris-Saclay\\
91128 PALAISEAU C\'edex,\\
FRANCE.\\
\end{tabular}\\
\end{document}